\documentclass[11pt]{article}
\usepackage{geometry}
\usepackage[utf8]{inputenc}
\usepackage{amssymb,amsmath,amsfonts,amsthm}
\usepackage{array}
\usepackage{mathrsfs}
\usepackage{graphicx,epsfig}
\usepackage{caption}
\usepackage{subcaption}

\usepackage{url}
\usepackage{lscape}
\usepackage{xspace}
\usepackage{algorithm}

\usepackage{algpseudocode}
\usepackage{hyperref}
\usepackage{mathtools}
\usepackage{multirow,multicol}
\usepackage{cleveref}
\usepackage{comment}
\usepackage{enumitem}

\usepackage{ulem}
\usepackage[usenames,dvipsnames]{color}
\usepackage{cancel}
\usepackage{mathdots}
\usepackage{diagbox}
\usepackage{blkarray}
\usepackage{extarrows}
\usepackage{arydshln}
\usepackage{mathtools}
\mathtoolsset{showonlyrefs=true}

\usepackage[title]{appendix}

\def\k{\mathbf{k}}

\def\r{\mathbf{r}}

\def\u{\mathbf{u}}
\def\v{\mathbf{v}}

\def\x{\mathbf{x}}

\newtheorem{thm}{Theorem}[section]
\newtheorem{remark}[thm]{Remark}
\newtheorem{lem}[thm]{Lemma}

\crefname{lem}{Lemma}{Lemmas}
\crefname{thm}{Theorem}{Theorems}
\crefname{cor}{Corollary}{Corollaries}
\crefname{pr}{Proposition}{Propositions}
\crefname{remark}{Remark}{Remarks}
\crefname{algorithm}{Algorithm}{Algorithms}
\crefname{appendix}{}{Appendices}
\crefname{section}{Section}{Sections}
\crefname{figure}{Figure}{Figures}

\numberwithin{equation}{section}

\crefname{table}{Table}{Tables}

\title{A Novel Computational and Analytical Framework for 2D Quasiperiodic Helmholtz Eigenvalue Problems via the Projection Method}

\author{
Teng-Chao Sun\footnote{School of Mathematics and Shing-Tung Yau Center, Southeast University, Nanjing 211189, People's Republic of China.}, \quad
Tiexiang Li\footnote{Corresponding author (txli@seu.edu.cn). School of Mathematics and Shing-Tung Yau Center, Southeast University, Nanjing 211189, People's Republic of China; Shanghai Institute for Mathematics and Interdisciplinary Sciences (SIMIS), Shanghai 200433, People's Republic of China.}, \quad
Wen-Wei Lin\footnote{Shanghai Institute for Mathematics and Interdisciplinary Sciences (SIMIS), Shanghai 200433, People's Republic of China; Department of Applied Mathematics, National Yang Ming Chiao Tung University, Hsinchu 300, Taiwan.} \quad
Xing-Long Lyu\footnote{School of Mathematical Sciences, Nanjing Normal University, Nanjing 210023, People's Republic of China}
}

\date{}

\begin{document}

\maketitle

\begin{abstract}
In this paper, we propose a spectral framework that embeds 1D and 2D quasiperiodic Helmholtz eigenvalue problems into higher-dimensional (2D and 4D) periodic spaces via the projection method \cite{jiang2014numerical, jiang2024numerical}. To effectively map the elevated high-dimensional states back to the original physical space, we establish a novel validation framework based on the weighted expectation of pointwise Rayleigh quotients. Supported by comprehensive error and spectral analysis, we demonstrate that the eigenvalues derived from this expectation align more authentically with the original quasiperiodic model, ultimately yielding a more appropriate and reliable eigenpair solution. Numerical experiments on  continuous media demonstrate that our approach offers an accurate, robust, and scalable tool for solving quasiperiodic Helmholtz eigenvalue problems.
\end{abstract}


\section{Introduction}

Quasiperiodic structures, characterized by long-range orientational order without any translational periodicity, have continuously fascinated the physics and mathematics communities since the groundbreaking discovery of quasicrystals by Shechtman \textit{et al.} \cite{shechtman1984metallic, levine1984quasicrystals}. Unlike conventional periodic system, quasiperiodic systems harbor a unique class of exotic wave phenomena, encompassing critical localized states, Cantor-set energy spectra, and Hofstadter butterfly patterns \cite{kohmoto1987localization, suto1989singular, bellissard1989spectral}. In the realm of electromagnetics and optics, photonic quasicrystals have emerged as a powerful platform for manipulating light propagation and controlling spontaneous emission \cite{vardeny2013optics, dal2003light}. By introducing quasiperiodic dielectric modulations, such materials exhibit complete photonic bandgaps, slow light effects, and robust topological boundary states \cite{chan1998photonic, freedman2006wave, vardeny2013optics}. Consequently, they hold tremendous potential in the design of high-$Q$ cavities, perfect absorbers, and optical sensors \cite{macia2006role, steinhardt1987physics}. The rigorous understanding and precise exploitation of these physical phenomena critically depend on the efficient solution of the governing partial differential equations endowed with quasiperiodic coefficients \cite{lifshitz2003quasicrystals}.

Solving quasiperiodic eigenvalue problems is computationally challenging due to the lack of spatial periodicity, which fundamentally invalidates the classical Bloch-Floquet theorem \cite{simon1982almost}. The traditional and most prevalent numerical approach is the crystalline approximant method (CAM), also referred to as the periodic approximation or supercell method \cite{goldman1993quasicrystals, tsunetsugu1991electronic}. By enforcing artificial periodic boundary conditions on a large truncated domain, the CAM approximates incommensurate frequencies using commensurate ones via simultaneous Diophantine approximation (SDA) \cite{davenport1946simultaneous, jiang2025approximation}. However, achieving high fidelity requires exponentially large supercells, which incurs severe finite-size effects and intractable computational costs, especially for multidimensional or highly discontinuous problems \cite{steurer2009crystallography}. To fundamentally circumvent these limitations, the cut-and-project method provides a rigorous geometric framework, revealing that any $d$-dimensional quasiperiodic structure can be obtained by projecting a high-dimensional ($n>d$) periodic lattice through an irrational orientation window \cite{de1981algebraic, duneau1985quasiperiodic}. This elegant mathematical insight shifts the perspective from analyzing non-periodic structures in physical space to analyzing perfectly periodic structures in higher dimensions \cite{avila2009ten}.

In recent years, the projection method (PM) has been developed as a natural and exact numerical framework for quasiperiodic systems \cite{jiang2014numerical}. By embedding the low-dimensional quasiperiodic functions into a higher-dimensional periodic space, the PM entirely bypasses the Diophantine approximation error and restores the applicability of fast Fourier transforms and spectral methods \cite{jiang2014numerical, jiang2024numerical, jiang2025projection}. Furthermore, a numerical approach capable of accurately recovering solutions for equations with continuous and low-regularity coefficients using a finite number of points has been recently proposed in \cite{jiang2024accurately}. Specifically addressing the Schr{\"o}dinger eigenvalue problem, \cite{jiang2025irrational} introduced an efficient preconditioning scheme based on irrational window filtering within the framework of the projection spectral method. Very recently, \cite{jiang2026quasiperiodic} theoretically established the spectral convergence for these projected eigenvalue problems at the continuous level. However, a practical numerical analysis detailing the exact mapping error from the truncated high-dimensional space back down to the original low-dimensional system is still not complete. Therefore, the convergence and numerical error associated with mapping the elevated eigenvalues by the projected spectral method back to the original quasiperiodic problem remains a concern.

Building upon the advancements in projection methods, we investigate the numerical methods for solving 1D and 2D Helmholtz eigenvalue problems (HEVPs) embedded into 2D and 4D spaces with quasiperiodic coefficients. In 2D space, the governing Helmholtz equations reduce to Helmholtz equations for transverse electric (TE) and transverse magnetic (TM) modes:
\begin{align}
    -\nabla\cdot(\varepsilon_\bot^{-1}(\mathbf{r})\nabla H_z(\mathbf{r}))&=\omega^2\mu_z(\mathbf{r})H_z(\mathbf{r}), \label{evp: TE mode} \\
    -\nabla\cdot(\mu_\bot^{-1}(\mathbf{r})\nabla E_z(\mathbf{r}))&=\omega^2\varepsilon_z(\mathbf{r})E_z(\mathbf{r}), \label{evp: TM mode}
\end{align}
respectively. Here, the coefficients $\varepsilon_z(\mathbf{r})$ and $\mu_z(\mathbf{r})$ are explicitly defined as scalar quasiperiodic functions with $\mathbf{r}\in\mathbb{R}^2$. The corresponding electric and magnetic fields are formulated as $E(\mathbf{r})=[E_x(\mathbf{r}),E_y(\mathbf{r}),0]^\mathsf{T}$ and $H(\mathbf{r})=[0,0,H_z(\mathbf{r})]^\mathsf{T}$ for the TE mode, and $E(\mathbf{r})=[0,0,E_z(\mathbf{r})]^\mathsf{T}$ and $H(\mathbf{r})=[H_x(\mathbf{r}),H_y(\mathbf{r}),0]^\mathsf{T}$ for the TM mode. The main contributions are summarized as follows:
\begin{itemize}
    \item Based on the projection method, we establish a unified spectral framework for embedding 1D/2D Helmholtz eigenvalue problems with continuous quasiperiodic coefficients into higher-dimensional periodic spaces.
    \item We propose a rigorous validation framework based on the weighted expectation of pointwise Rayleigh quotients to establish numerical consistency, ensuring the resulting eigenpairs authentically represent the original quasiperiodic system.
    \item We provide a comprehensive theoretical error and spectral analysis for the pointwise Rayleigh quotients, proving that the computed eigenvalues from the embedded periodic space accurately approximate the infinite-dimensional quasiperiodic spectrum.
\end{itemize}


\paragraph{Notation.}
We use $\imath=\sqrt{-1}$ for the imaginary unit. Bold lowercase letters are used for vectors, such as the physical variable $\r$ and the higher-dimensional periodic variable $\x$. The weighted expectation with respect to the probability distribution $\mathrm{p}$ is written as $\mathbb E_{\mathrm p}[\cdot]$. Superscripts $\boldsymbol{u}^{(1)}$ and $\boldsymbol{u}^{(2)}$ are used to distinguish quantities associated with the 1D quasiperiodic system and the 2D quasiperiodic system, respectively.

\section{1D HEVPs via Associated 2D Embedded EVPs} \label{section: 1Dto2D}

In this section, to illustrate our idea more clearly, we begin with the simplest 1D HEVPs arising from the TM-mode in \eqref{evp: TM mode} with $\mu(r) \equiv 1$
\begin{equation} \label{evp: 1D HEVPs}
- \frac{d^2}{dr^2} u(r) = \lambda \, \varepsilon(r) u(r),
\quad r \in \mathbb{R},
\end{equation}
where $u(r) = E_z(r)$ is a scalar field, $\varepsilon(r)$ is a quasiperiodic function of $r$, and $\lambda$ denotes the corresponding eigenvalue.

A direct finite-difference (FD) discretization of \eqref{evp: 1D HEVPs} leads to the following infinite-dimensional generalized eigenvalue problem (GEVP):
\begin{subequations} \label{total: discrete form of 2D}
\begin{equation} \label{df: 1D HEVP}
\mathcal{A} \boldsymbol{u} = \lambda \mathcal{B} \boldsymbol{u},
\end{equation}
where $\boldsymbol{u}=[\ldots,u_{-1},u_0,u_1,\ldots]^\mathsf{T}$ is the eigenvector corresponding to the eigenvalue $\lambda\in\mathbb{R}$. The operators $\mathcal{A}$ and $\mathcal{B}$ in \eqref{df: 1D HEVP} are given by the discrete Laplacian and a diagonal coefficient operator:
\begin{equation}\label{df: infinite matrix}
\mathcal{A} = \frac{1}{h^2}\,\mathrm{tridiag}(c_{i-1},\, a_i,\, b_i)_{i=0}^{\pm\infty},
\quad
\mathcal{B} = \mathrm{diag}(\varepsilon_i)_{i=0}^{\pm \infty},
\end{equation}
where $(c_{i-1},a_i,b_i)\equiv(-1,2,-1)$ are the subdiagonal, diagonal, and superdiagonal entries in the $i$th row, respectively, $\varepsilon_i = \varepsilon(r_i)$ is the $i$th diagonal entry and $h$ is the mesh size of the discretization.
\end{subequations}

Since the eigenvalue problem \eqref{total: discrete form of 2D} is infinite-dimensional, a direct numerical treatment requires truncation. One possible approach is to restrict the problem to a large but finite interval and impose suitable boundary conditions, such as Dirichlet boundary conditions. This leads to a finite-dimensional GEVP of size $(2n+1)\times(2n+1)$
\begin{subequations} \label{total: truncated scheme}
    \begin{equation} \label{df: truncated scheme}
    A_{\mathrm{t}} \mathbf{u}_{\mathrm{t}} = \lambda_{\mathrm{t}} B_{\mathrm{t}} \mathbf{u}_{\mathrm{t}},
    \end{equation}
where $\mathbf{u}_{\mathrm{t}} = [u_{-n}, \dots, u_0, \dots, u_{n}]^\mathsf{T}$, and
\begin{align}\label{df: truncated AB}
&A_{\mathrm{t}} = \frac{1}{h^2}\,\mathrm{tridiag}(c_{i-1},\, a_i,\, b_i)_{i=0}^{\pm n}
\in \mathbb{R}^{(2n+1) \times (2n+1)}, \quad  B_{\mathrm{t}} = \mathrm{diag}(\varepsilon_{-n}, \dots, \varepsilon_0, \dots, \varepsilon_{n}).
\end{align}
\end{subequations}

While the truncated system can partially approximate the original quasiperiodic problem, the computed eigenvalues $\lambda_{\mathrm{t}}$ often suffer from significant numerical errors due to the SDA \cite{davenport1946simultaneous, jiang2025approximation}. Since $\varepsilon(r)$ in \eqref{evp: 1D HEVPs} is quasiperiodic, it can, in general, be embedded into a periodic 2-torus $\mathcal{T}^2$ with period $\mathbf{T} = (T_1, T_2)$ via a projection matrix $\, P = \left(p_1,  p_2\right)^\mathsf{T}$  with the components $p_1$ and $p_2$ being rationally independent, i.e. $\alpha_1p_1 + \alpha_2p_2 = 0, \, \forall \,\alpha_1, \alpha_2 \in \mathbb{Q}$ if and only if $\alpha_1 = \alpha_2 = 0.$ Motivated by the projection method and the associated theory in \cite{jiang2014numerical, jiang2024numerical},
we define the map
\begin{equation} \label{eq: projection relationship}
\mathbf{x} = P r =
\begin{bmatrix}
p_1 \\ p_2
\end{bmatrix} r,
\quad
\mathcal{E}(x_1+T_1, \, x_2 + T_2) = \mathcal{E}(x_1, x_2),
\quad
\varepsilon(r) = \mathcal{E}(P r).
\end{equation}

Furthermore, Let $U(\x)$ be a function defined on $\mathcal{T}^2$ such that
\begin{equation} \label{eq: projection field}
U(\mathbf{x}) = U(Pr) = u(r) , \quad P = \left(p_1,  p_2\right)^\mathsf{T} .
\end{equation}
Along the embedded line $\mathbf{x}=Pr$, the chain rule on $\frac{d}{dr} U(\mathbf{x})$ and $ \frac{d^2}{dr^2} U(\mathbf{x})$ yields
\begin{align} \label{eq: chain rule difference}
\frac{d}{dr} U(\mathbf{x})
&=
\left(
\frac{d x_1}{d r} \frac{\partial}{\partial x_1}
+
\frac{d x_2}{d r} \frac{\partial}{\partial x_2}
\right) U(\mathbf{x}) \notag \\
&=
\left(
p_1 \frac{\partial}{\partial x_1}
+
p_2 \frac{\partial}{\partial x_2}
\right) U(\mathbf{x}) \notag \\
&=
P^\mathsf{T} \nabla_{\mathbf{x}} U(\mathbf{x})
=: \nabla_P U(\mathbf{x}),
\end{align}
and
\begin{equation} \label{eq: chain rule ord2 diff}
\frac{d^2}{dr^2} U(\mathbf{x})
=
\left( p_1 \frac{\partial}{\partial x_1}
+
p_2 \frac{\partial}{\partial x_2}\right)^2 U(\mathbf{x})
=
\nabla_P^2 U(\mathbf{x}),
\end{equation}
where $\nabla_P^2$ denotes the second-order directional derivative along the projection direction. Hence, combining \eqref{eq: chain rule ord2 diff} with \eqref{eq: projection relationship}--\eqref{eq: projection field}, the associated 2D embedded eigenvalue problem of \eqref{evp: 1D HEVPs} becomes
\begin{equation} \label{evp: 1Dto2D continous projection}
- \nabla_P^2 U(\mathbf{x})
=
\lambda \, \mathcal{E}(\mathbf{x}) U(\mathbf{x}),
\quad
\mathbf{x} \in \mathcal{T}^2 \subseteq \mathbb{R}^2.
\end{equation}

We employ the Fourier-pseudospectral method \cite{Orszag1972Comparison} to solve the embedded eigenvalue problem \eqref{evp: 1Dto2D continous projection} on the 2-torus $\mathcal{T}^2$. Conventional real-space discretizations, such as FD, struggle in this higher-dimensional framework due to a fundamental irrational-direction bottleneck. Specifically, the physical 1D structure is embedded in $\mathcal{T}^2$ along a direction with irrational slope. A uniform Cartesian FD grid cannot align exactly with this incommensurate projection, and approximating it invariably induces severe ``staircase errors'' that pollute the wave dynamics \cite{taflove2005computational}. In contrast, the Fourier--pseudospectral method expands the solution in global Fourier bases, so the irrational directional derivatives can be evaluated analytically in reciprocal space, thereby eliminating grid-alignment errors \cite{Orszag1972Comparison, jiang2014numerical}.

Since $\mathcal{E}(\mathbf{x})$ is periodic on $\mathcal{T}^2$, we assume that the solution $U(\x)$ of \eqref{evp: 1Dto2D continous projection} has the Bloch-Floquet form \cite{joannopoulos2011photonic}
\begin{equation} \label{eq: Bloch expand}
U(\mathbf{x}) = e^{\imath \mathbf{k} \cdot \mathbf{x}} U_p(\mathbf{x}),
\end{equation}
 where $U_p$ is $\mathbf{T}$-periodic and $\mathbf{k}$ is the Bloch wave vector in the embedded reciprocal space. It plays the same role as the Bloch wave vector in periodic photonic crystals, but it is associated with the higher-dimensional periodic formulation rather than the original low-dimensional physical domain. Since $\mathcal{E}(\x)$ and $U_p(\x)$ are $\mathbf{T}$-periodic functions on the torus $\mathcal{T}^2$, while $U(\mathbf{x})$ is the corresponding $\mathbf{k}$-quasiperiodic function, we can expand them into Fourier series over the reciprocal space as

\begin{subequations} \label{total: 2D project fourier span}
    \begin{equation} \label{eq: fourier span}
\mathcal{E}(\mathbf{x}) = \sum_{\boldsymbol{\xi} \in \mathcal{J}} \mathcal{E}_{\boldsymbol{\xi}} e^{\imath \boldsymbol{\xi} \cdot \mathbf{x}},
\quad
U_p(\mathbf{x}) = \sum_{\boldsymbol{\xi} \in \mathcal{J}} U_{\boldsymbol{\xi}} e^{\imath \boldsymbol{\xi} \cdot \mathbf{x}},
\end{equation}

\begin{equation} \label{eq: field fourier span}
U(\mathbf{x}) = e^{\imath \mathbf{k} \cdot \mathbf{x}}U_p(\mathbf{x}) = \sum_{\boldsymbol{\xi} \in \mathcal{J}} U_{\boldsymbol{\xi}} e^{\imath (\mathbf{k} + \boldsymbol{\xi}) \cdot \mathbf{x}},
\end{equation}
where the index set $\mathcal{J}$ is defined by
\begin{align} \label{set: reciprocal space}
\mathcal{J} = \left\{ \left(\frac{2\pi}{T_1}i,\frac{2\pi}{T_2}j \right) \,\middle|\, i,j \in \left\{-N, \ldots, N-1\right\} \right\}, \quad \text{with} \,\, \#\mathcal{J} = (2N)^2.
\end{align}
\end{subequations}

Applying the projected derivative defined in \eqref{eq: chain rule difference}--\eqref{eq: chain rule ord2 diff} to \eqref{eq: field fourier span}, we get
\begin{align} \label{eq: ord2 diff derive}
    \nabla_{P}^{2} U(\mathbf{x}) &= \nabla_{P} \left( \sum_{\boldsymbol{\xi} \in \mathcal{J}} U_{\boldsymbol{\xi}} P^{\mathsf{T}} \nabla_{\mathbf{x}} e^{\imath (\mathbf{k}+\boldsymbol{\xi}) \cdot \mathbf{x}} \right) \notag  = \nabla_{P} \left( \sum_{\boldsymbol{\xi} \in \mathcal{J}} U_{\boldsymbol{\xi}} \left( \imath P^{\mathsf{T}} \left(\mathbf{k}+\boldsymbol{\xi}\right) \right) e^{\imath(\mathbf{k}+\boldsymbol{\xi}) \cdot \mathbf{x}} \right) \notag \\
    &= -\sum_{\boldsymbol{\xi} \in \mathcal{J}} \left( P^{\mathsf{T}} (\mathbf{k}+\boldsymbol{\xi}) \right)^{2} U_{\boldsymbol{\xi}} e^{\imath(\mathbf{k}+\boldsymbol{\xi}) \cdot \mathbf{x}}.
\end{align}
The multiplication term leads to a discrete convolution in Fourier space. More precisely, within the truncated index set $\mathcal{J}$, the resulting index is understood modulo the Fourier grid, i.e., $\boldsymbol{\xi} = \boldsymbol{\xi}^\prime + \boldsymbol{\xi}^{\prime\prime}
\,  \pmod{2N},$ we obtain
\begin{align} \label{eq: Convolution of Fourier coef}
\mathcal{E}(\mathbf{x}) U(\mathbf{x})
&=
\left(
\sum_{\boldsymbol{\xi}^{\prime \prime} \in \mathcal{J}} \mathcal{E}_{\boldsymbol{\xi}^{\prime \prime}} e^{\imath \boldsymbol{\xi}^{\prime \prime} \cdot \mathbf{x}}
\right)
\left(
\sum_{\boldsymbol{\xi}^{\prime} \in \mathcal{J}} U_{\boldsymbol{\xi}^{\prime}} e^{\imath (\mathbf{k} + \boldsymbol{\xi}^{\prime}) \cdot \mathbf{x}}
\right) \notag =
\sum_{\boldsymbol{\xi}^{\prime \prime} \in \mathcal{J}} \sum_{\boldsymbol{\xi}^{\prime} \in \mathcal{J}}\mathcal{E}_{\boldsymbol{\xi}^{\prime \prime}}U_{\boldsymbol{\xi}^{\prime}} e^{\imath (\mathbf{k} + \boldsymbol{\xi}^{\prime \prime} + \boldsymbol{\xi}^{\prime}) \cdot \mathbf{x}} \notag \\
&=
\sum_{\boldsymbol{\xi} \in \mathcal{J}}
\left(
\sum_{\boldsymbol{\xi}^{\prime} \in \mathcal{J}}
\mathcal{E}_{\boldsymbol{\xi} - \boldsymbol{\xi}^{\prime}} U_{\boldsymbol{\xi}^{\prime}}
\right)
e^{\imath (\mathbf{k} + \boldsymbol{\xi}) \cdot \mathbf{x}}.
\end{align}

Let $\mathcal{J}_\v = $ vec($\mathcal{J}$) denote the vectorized index set of $\mathcal{J}$. By the orthogonality of the Fourier basis, substituting \eqref{eq: ord2 diff derive} and \eqref{eq: Convolution of Fourier coef} into \eqref{evp: 1Dto2D continous projection}  yields the discrete system
\begin{equation} \label{df: Spectral projection term}
\left( P^\mathsf{T} (\mathbf{k} + \boldsymbol{\xi}) \right)^2 U_{\boldsymbol{\xi}}
=
\widetilde{\lambda}
\sum_{\boldsymbol{\xi}^{\prime} \in \mathcal{J}_\v}
\mathcal{E}_{\boldsymbol{\xi} - \boldsymbol{\xi}^{\prime}} U_{\boldsymbol{\xi}},
\end{equation}
where $\widetilde{\lambda}$ denotes the eigenvalue of the truncated Fourier-discretized system, distinguished from the continuous eigenvalue $\lambda$.

Then, for a given $\k$, we obtain the discrete GEVP corresponding to \eqref{evp: 1Dto2D continous projection}:
\begin{equation} \label{df: Spectral Method 2D}
K U= \widetilde{\lambda} M U,  \quad \text{with} \quad U = \left(U_{\boldsymbol{\xi}}\right)_{\boldsymbol{\xi} \in \mathcal{J}_\v},
\end{equation}
where $K = \left[ K_{\boldsymbol{\xi}\boldsymbol{\xi}} \right]_{\boldsymbol{\xi} \in \mathcal{J}_\v} $ is a diagonal matrix with $K_{\boldsymbol{\xi}\boldsymbol{\xi}} = \left( P^\mathsf{T} (\mathbf{k} + \boldsymbol{\xi}) \right)^2$  and $M$ is a block Toeplitz matrix with $ M_{\boldsymbol{\xi}\boldsymbol{\xi}^\prime} = \mathcal{E}_{\boldsymbol{\xi} - \boldsymbol{\xi}^\prime}$, for $\, \boldsymbol{\xi}, \boldsymbol{\xi}\prime \in \mathcal{J}_\v.$

For any  fixed Bloch vector $\mathbf{k}$, once the discrete eigenpair $(\widetilde{\lambda},U)$ has been computed, the corresponding eigenfunction can be projected back to the 1D physical domain by using \eqref{eq: projection field} and \eqref{eq: field fourier span}. This gives
\begin{equation} \label{sol: spectral projection}
\widehat{u}(r)
= U(P r)
= \sum_{\boldsymbol{\xi} \in \mathcal{J}_\v}
U_{\boldsymbol{\xi}} e^{\imath \left(P^\mathsf{T}(\mathbf{k} + \boldsymbol{\xi})\right) r},
\end{equation}
for 1D HEVP \eqref{evp: 1D HEVPs}.  Note that since the components of $P=(p_1,p_2)^{\mathsf{T}}$ are rationally independent, $\widehat{u}(r)$ is generally quasiperiodic rather than periodic.

To assess whether the computed eigenvalue $\widetilde{\lambda}$ from \eqref{df: Spectral Method 2D} by the Fourier-pseudospectral method is a good approximation to $\lambda$ in \eqref{evp: 1D HEVPs}, we use the 1D infinite-dimensional FD discretization \eqref{total: discrete form of 2D} together with a pointwise Rayleigh quotient (RQ) to perform validation over a sufficiently large set of sampling points. Specifically, we denote the discretization of the reconstructed eigenfunction $\widehat{u}(r)$ over the entire 1D domain $\mathbb{R}$ by the infinite vector $\widehat{\boldsymbol{u}}$ and evaluate its accuracy on a set of $2n+1$ sampling points
\begin{align} \label{def: 2D sampling points}
    \mathcal{R} = \{ r_i \mid r_i = r_0 + i h,\ r_0\in\mathbb{R}, i\in\mathcal{S}(n)\}, \quad \mathcal{S}(n) = \{-n, -n+1, \ldots, n\},
\end{align}
with $h$ being a sufficiently small mesh size. Restricting the action of the infinite-dimensional operators to the sampled indices gives
\begin{subequations} \label{total: finite matrix derive}
    \begin{align}
        &(\mathcal{A}\widehat{\boldsymbol{u}} )|_{\mathcal{S}(n)}
        =
        \mathcal{A}_{\{\mathcal{S}(n),\vcenter{\hbox{$\cdot$}} \}}\widehat{\boldsymbol{u}}|_{\mathcal{S}(n+1)}
        =\widehat{A}\widehat{\boldsymbol{u}}|_{\mathcal{S}(n+1)} \equiv \widehat{A}\widehat{\u}, \label{dv: matrixA}\\
        &(\mathcal{B}\widehat{\boldsymbol{u}} )|_{\mathcal{S}(n)}
        = \mathcal{B}_{\{\mathcal{S}(n),\vcenter{\hbox{$\cdot$}} \}}\widehat{\boldsymbol{u}}|_{\mathcal{S}(n)}
        =\widehat{B}\widehat{\boldsymbol{u}}|_{\mathcal{S}(n)}\equiv \widehat{B}\widehat{\u}_{\mathrm{o}},\label{dv: matrixB}
    \end{align}
\end{subequations}
where
\begin{equation} \label{df: truncted matrices of 2D}
    \widehat{A} = \mathcal{A}_{\{\mathcal{S}(n), \mathcal{S}(n+1)\}} \in \mathbb{R}^{(2n+1) \times (2n+3)},
\quad
\widehat{B} = \mathcal{B}_{\{\mathcal{S}(n), \mathcal{S}(n)\}}
\end{equation}
are the cropped tridiagonal and diagonal matrices extracted from $\mathcal{A}$ and $\mathcal{B}$, respectively, and
\begin{equation} \label{df: discretized vector}
     \widehat{\mathbf{u}} = [\widehat{u}_{-n-1}, \ldots, \widehat{u}_{n+1}]^\mathsf{T}, \quad \widehat{\mathbf{u}}_{\mathrm{o}} = [\widehat{u}_{-n}, \ldots, \widehat{u}_{n}]^\mathsf{T}, \quad \widehat{u}_i = \widehat{u}(r_i) , \quad \varepsilon_i = \varepsilon(r_i) .
\end{equation}

The pointwise RQs are then defined by the componentwise ratio between $\widehat{A}\widehat{\mathbf{u}}$ and $\widehat{B}\widehat{\mathbf{u}}_{\mathrm{o}}$. Their weighted expectation is used as a consistency-based estimator for the reconstructed eigenfunction in the original 1D discrete system:
\begin{equation} \label{def: expectation of ratio}
\widehat{\lambda}
\equiv
\mathbb{E}_{\mathrm{p}} \left[
\left(\widehat{A}\widehat{\mathbf{u}}\right)\oslash\left(\widehat{B}\widehat{\mathbf{u}}_{\mathrm{o}}\right)
\right] = \mathbb{E}_{\mathrm{p}}\left[\Bigg\{\frac{\left(\widehat{A}\widehat{\u}\right)_i}{\left(\widehat{B}\widehat{\u}_{\mathrm{o}}\right)_i} \Bigg\}_{i=0}^{\pm n}\right]
\approx \widetilde{\lambda},
\end{equation}
where $\oslash\,$ denotes Hadamard quotient and $\mathrm{p} = \{p_i = \frac{\varepsilon_i \widehat{u}_i^2}{\widehat{\mathbf{u}}_{\mathrm{o}}^\mathsf{T} \widehat{B}\widehat{\mathbf{u}}_{\mathrm{o}}}\}_{i=0}^{\pm n}$ is the weight associated with the sampled points.


This section has reformulated formulated the 1D quasiperiodic HEVP as a 2D periodic embedded problem and derived its Fourier spectral discretization. The computed
eigenfunction is then reconstructed on the physical line, where the pointwise RQ provides an eigenvalue estimator for the original low-dimensional problem. The theoretical justification of this estimator, including its consistency and stability, will be presented in Section~\ref{section: Error analysis}.

\section{2D HEVPs via Associated 4D Embedded EVPs} \label{section: 2Dto4D}

Building on the 1D-to-2D construction in the previous section, we now extend the projection framework to the 2D-to-4D setting. In this section we consider the 2D TM-mode in \eqref{evp: TM mode} with $\mu_{\bot}(\r) \equiv 1$

\begin{equation} \label{evp: 2D TM mode}
- \nabla\!\cdot\! \nabla u(\mathbf{r})
= \lambda \, \varepsilon(\mathbf{r}) u(\mathbf{r}),
\quad
\mathbf{r} = (r_1, r_2)^\mathsf{T} \in \mathbb{R}^2,
\end{equation}
where $u(\mathbf{r}) = E_z(\r)$ is a scalar field, $\varepsilon(\mathbf{r}) = \varepsilon_z(\r)$ is a quasiperiodic function and $\lambda $ is the corresponding eigenvalue. As in \eqref{total: discrete form of 2D}, we can discrete \eqref{evp: 2D TM mode} by FD method and derive an infinite-dimensional eigenvalue problem
\begin{equation} \label{df: 2D TMmode}
\mathcal{A}^{(2)} \boldsymbol{u}^{(2)} = \lambda \mathcal{B}^{(2)} \boldsymbol{u}^{(2)}.
\end{equation}
The operators $\mathcal{A}^{(2)}$ and $\mathcal{B}^{(2)}\,$are given by
\begin{subequations} \label{total: discrete infinite 2D}
    \begin{equation} \label{df: 2D diff scheme}
    \mathcal{A}^{(2)} = \mathcal{I} \otimes \mathcal{A} + \mathcal{A} \otimes \mathcal{I},
    \quad
    \mathcal{B}^{(2)} = \mathrm{blkdiag}(\dots, \boldsymbol{\varepsilon}_{-1}, \boldsymbol{\varepsilon}_{0}, \boldsymbol{\varepsilon}_{1}, \dots),
    \end{equation}
    in which $\mathcal{A}$ is given in \eqref{df: infinite matrix}, $\mathcal{I}$ denotes the identity operator,  and `blkdiag' denotes a block diagonal matrix. Each block is define as $\boldsymbol{\varepsilon}_i = \text{diag}(\ldots, \varepsilon_{i,-1}, \varepsilon_{i,0}, \varepsilon_{i,1}, \ldots)$ and $ \varepsilon_{i,j} = \varepsilon(\mathbf{r}_{i,j}), \,i,j = 0, \pm 1, \cdots \pm \infty,$ with $\mathbf{r}_{i,j}$ being the mesh points. The vectorized unknown is
    \begin{equation} \label{set: 2D infinite discrete vector}
    \boldsymbol{u}^{(2)} = [\ldots, \boldsymbol{u}_{-1}^\mathsf{T}, \boldsymbol{u}_{0}^\mathsf{T}, \boldsymbol{u}_{1}^\mathsf{T}, \dots]^\mathsf{T} \quad \text{with} \quad \boldsymbol{u}_i = [\dots, u_{i,-1}, u_{i,0}, u_{i,1}, \ldots]^\mathsf{T},
    \end{equation}
\end{subequations}
    for $i = 0, \pm1, \ldots, \pm \infty.$

In the following, we employ the projection method to embed \eqref{evp: 2D TM mode} into a 4D eigenvalue problem. Since $\varepsilon(\mathbf{r})$ in \eqref{evp: 2D TM mode} is quasiperiodic, it can, in general, be embedded into a 4-torus $\mathcal{T}^4$ with period $\mathbf{T} = (T_1, \dots, T_4)$ via a projection matrix

\begin{subequations} \label{total: 4D projected matrices}
        \begin{equation} \label{def: 4D projected matrix}
    P = \begin{bmatrix}
        p_{11} & p_{12} & p_{13} & p_{14} \\
        p_{21} & p_{22} & p_{23} & p_{24}
    \end{bmatrix}^{\mathsf{T}} \in \mathbb{R}^{4 \times 2},
    \end{equation}
   where the rows of $P$, denoted by $\mathbf{p}_j=(p_{j1},p_{j2})$ for $j=1,\ldots,4$, are assumed to be rationally independent, as in Section~\ref{section: 1Dto2D}. The projected coordinate is
    \begin{equation} \label{eq: periodic projection variation}
    \mathbf{x} = P \mathbf{r} \in \mathcal{T}^4,
    \quad
    \mathcal{E}(\mathbf{x} + \mathbf{T}) = \mathcal{E}(\mathbf{x}) = \mathcal{E}(P\mathbf{r}) =
    \varepsilon(\mathbf{r}) .
    \end{equation}
\end{subequations}
Let the lifted function $U(\mathbf{x})$ satisfy
\begin{equation} \label{eq: projected relation}
U(\mathbf{x}) = U(P\mathbf{r}) = u(\mathbf{r}),
\quad \mathbf{r} = (r_1, r_2)^\mathsf{T} \in \mathbb{R}^2.
\end{equation}
Applying the chain rule yields
\begin{align} \label{eq: projected operator 4D}
\nabla_{\mathbf{r}}\!\cdot\! \nabla_{\mathbf{r}} U(\mathbf{x})
&= \nabla_{\mathbf{r}} \!\cdot\! \left( \nabla_P U(\mathbf{x}) \right)
= \nabla_{\mathbf{r}} \!\cdot\! \left( (P^\mathsf{T} \nabla_{\mathbf{x}}) U(\mathbf{x}) \right) \notag \\
&= (\nabla_{\mathbf{x}}^\mathsf{T} P)(P^\mathsf{T} \nabla_{\mathbf{x}}) U(\mathbf{x})
\equiv \nabla_P^2 U(\mathbf{x}).
\end{align}
Combining \eqref{eq: projected operator 4D} with \eqref{total: 4D projected matrices}, the original problem \eqref{evp: 2D TM mode} can be reformulated as the 4D embedded eigenvalue problem
\begin{equation} \label{evp: 2Dto4D continous projection}
- \nabla_P^2 U(\mathbf{x})
= \widetilde{\lambda} \, \mathcal{E}(\mathbf{x}) U(\mathbf{x}),
\quad \mathbf{x} \in \mathcal{T}^4.
\end{equation}

We now utilize the Fourier pseudospectral method \cite{Orszag1972Comparison} to discrete \eqref{evp: 2Dto4D continous projection}. Extending the construction in Section~\ref{section: 1Dto2D}, we introduce the reciprocal index set
\begin{align} \label{set: 4D reciprocal index}
\mathcal{J} = \left\{ \left(\frac{2\pi}{T_1}i,\frac{2\pi}{T_2}j,\frac{2\pi}{T_3}k,\frac{2\pi}{T_4}\ell\right) \mid i,j,k,\ell \in \left\{-N, \ldots, N-1\right\}  \right\},
\end{align}
with $\#\mathcal{J} = (2N)^4$. Letting $\mathcal{J}_\v = \mathrm{vec}(\mathcal{J})$, for a given $\k \in \mathbb{R}^4$ we obtain the GEVP corresponding to \eqref{evp: 2Dto4D continous projection}
\begin{subequations} \label{total: 4D spectral}
    \begin{equation} \label{df: 4D spectral}
K U = \widetilde{\lambda} M U,
\quad \text{with} \,\, U = (U_{\boldsymbol{\xi}})_{\boldsymbol{\xi} \in \mathcal{J}_\v},
\end{equation}
where $K$ is diagonal and $M$ is a block Toeplitz matrix with
\begin{equation}
    K_{\boldsymbol{\xi}\boldsymbol{\xi}} =
(\mathbf{k} + \boldsymbol{\xi})^\mathsf{T} P P^\mathsf{T} (\mathbf{k} + \boldsymbol{\xi}), \quad M_{\boldsymbol{\xi}\boldsymbol{\xi}^\prime} = \mathcal{E}_{\boldsymbol{\xi} - \boldsymbol{\xi}^\prime},
\end{equation}
\end{subequations}
 for ${\boldsymbol{\xi}, \boldsymbol{\xi}^\prime} \in \mathcal{J}_\v.$

The computed eigenfunction of \eqref{total: 4D spectral} can be projected back to the 2D physical domain as

\begin{equation} \label{sol: 4D projected function}
\widehat{u}(\mathbf{r})
= U(P\mathbf{r})
= \sum_{\boldsymbol{\xi} \in \mathcal{J}_\v}
U_{\boldsymbol{\xi}}
e^{\imath (P^\mathsf{T} (\mathbf{k} + \boldsymbol{\xi})) \cdot \mathbf{r}}.
\end{equation}

To assess the accuracy of $\widetilde{\lambda}$, we employ the 2D FD operator \eqref{total: discrete infinite 2D} and evaluate pointwise RQs on a large sampling set $\mathcal{R}^{(2)} = \{ \mathbf{r}_{ij} \mid \mathbf{r}_{ij} = \mathbf{r}_{00} + [ih, jh]^\mathsf{T},\, \r_{00}\in\mathbb{R}^2, \, (i,j) \in \mathcal{S}(n) \times \mathcal{S}(n)\},$ where $ \mathcal{S}(n) $ is defined in \eqref{def: 2D sampling points}. Denoting the sampled indices by $\mathcal{S}^{(2)}(n) = \text{vec}( \mathcal{S}(n) \times \mathcal{S}(n))$ with `vec' representing the vectorization of a 2D index but through $\pm \infty$, we have

\begin{subequations} \label{total: finite matrix derive 4D}
    \begin{align}
        (\mathcal{A}^{(2)}\widehat{\boldsymbol{u}}^{(2)} )|_{\mathcal{S}^{(2)}(n)}
        =&
        \mathcal{A}^{(2)}_{\{\mathcal{S}^{(2)}(n),\cdot\}}\widehat{\boldsymbol{u}}^{(2)}|_{\mathcal{S}^{(2)}(n+1)}
        =\widehat{A}^{(2)}\widehat{\boldsymbol{u}}^{(2)}|_{\mathcal{S}^{(2)}(n+1)} \equiv \widehat{A}^{(2)}\widehat{\u}^{(2)}, \label{dv: matrixA 4D}\\
        (\mathcal{B}^{(2)}\widehat{\boldsymbol{u}}^{(2)} )|_{\mathcal{S}^{(2)}(n)}
        =& \mathcal{B}^{(2)}_{\{\mathcal{S}^{(2)}(n),\cdot\}}\widehat{\boldsymbol{u}}^{(2)}|_{\mathcal{S}^{(2)}(n)}
        =\widehat{B}^{(2)}\widehat{\boldsymbol{u}}^{(2)}|_{\mathcal{S}^{(2)}(n)}\equiv \widehat{B}^{(2)}\widehat{\u}_{\mathrm{o}}^{(2)},\label{dv: matrixB 4D}
    \end{align}
\end{subequations}
where
\begin{equation} \label{df: matrices of 4D}
    \widehat{A}^{(2)} = \mathcal{A}^{(2)}_{\{\mathcal{S}^{(2)}(n), \mathcal{S}^{(2)}(n+1)\}} \in \mathbb{R}^{(2n+1)^2 \times (2n+3)^2},
\quad
\widehat{B} = \mathcal{B}^{(2)}_{\{\mathcal{S}^{(2)}(n),\mathcal{S}^{(2)}(n)\}}
\end{equation}
are cropped from $\mathcal{A}^{(2)}$ and $\mathcal{B}^{(2)}$, respectively.

The approximate eigenvalue for 2D HEVP \eqref{evp: 2D TM mode} is estimated as the expectation of the Hadamard quotient of vectors $\widehat{A}^{(2)} \widehat{\mathbf{u}}^{(2)}$ and $\widehat{B}^{(2)} \widehat{\mathbf{u}}^{(2)}_{\mathrm{o}}$

\begin{equation} \label{def: dot rate of 2D}
\widehat{\lambda}
\equiv
\mathbb{E}_{\mathrm{p}} \left[
\left(\widehat{A}^{(2)}\widehat{\mathbf{u}}^{(2)}\right)\oslash\left(\widehat{B}^{(2)}\widehat{\mathbf{u}}^{(2)}_{\mathrm{o}}\right)
\right]
=
\mathbb{E}_{\mathrm{p}}
\left[
\{r_{s}\}_{s\in\mathcal{S}^{(2)}(n)}
\right]
\approx \widetilde{\lambda},
\end{equation}
where $\mathrm{p} = \left\{ p_{s} = \frac{\varepsilon_{s} \widehat{u}_{s}^2}{\widehat{\mathbf{u}}_{\mathrm{o}}^{(2)\mathsf{T}} \widehat{B}^{(2)} \widehat{\mathbf{u}}^{(2)}_{\mathrm{o}}} \right\}_{s\in\mathcal{S}^{(2)}(n)}$ with $\varepsilon_{s} = \widehat{B}^{(2)}_{\{s,s\}}$ and $ r_{s} = \left(\widehat{A}^{(2)} \widehat{\u}^{(2)}\right)_{s} \Bigg/\left(\widehat{B}^{(2)} \widehat{\mathbf{u}}^{(2)}_{\mathrm{o}}\right)_{s} $.  An error analysis for pointwise RQs $\{r_{s}\} $ can be found in the next section.

\cref{alg: projected spectral method} summarizes the proposed projection-based spectral framework using the representative 2D-to-4D quasiperiodic embedding. We first solve the embedded high-dimensional eigenvalue problem by a Fourier-pseudospectral discretization \eqref{df: 4D spectral}, and then use the computed eigenfunction $U$ to reconstruct the low-dimensional field $\widehat{u}$. The pointwise RQ estimator $\widehat{\lambda}$ is subsequently computed from $\widehat u$ on the low-dimensional FD grid as in \eqref{def: dot rate of 2D}.

In practical computation, we exploit the matrix structure resulting from the Fourier-pseudospectral discretization to accelerate the solution of the GEVP \eqref{df: 4D spectral}. The stiffness matrix $K$ is diagonal and is assembled explicitly in sparse form, while the mass matrix $M$ has a convolution structure and is applied in a matrix-free manner using FFTs. This reduces each application of $M$ from the $\mathcal{O} (N_F^2)$ cost of a dense matrix-vector multiplication to $\mathcal{O} (N_F\log N_F)$, where $N_F=\#\mathcal{J}$ is the total number of  Fourier bases. Since $\mathcal{E}(\mathbf{x})>0$, $M$ is symmetric positive definite. To compute the smallest positive eigenvalues of $KU=\widetilde{\lambda}MU$, we solve the inverse generalized problem $MU=\mu KU$ with $\mu=\widetilde{\lambda}^{-1}$ and extract the largest eigenvalues $\mu$ using a matrix-free iterative eigensolver.

The additional cost of computing the pointwise RQ estimator \eqref{def: dot rate of 2D} is negligible compared with the cost of the eigensolver. Although a sufficiently large number of physical sampling points $n$ is used to make the weighted expectation numerically stable, the evaluation of the pointwise RQs only requires the action of the cropped low-dimensional FD operator, together with componentwise products, quotients, and summations on the sampled vectors. Therefore, compared with the iterative solution of the embedded high-dimensional eigenvalue problem, the cost of obtaining $\widehat{\lambda}$ is very small. In this sense, the pointwise RQ procedure is a cheap but valuable  step for obtaining a more appropriate low-dimensional eigenvalue for the original quasiperiodic system.

\begin{algorithm}[htbp]
\caption{Projected Fourier-pseudospectral method for quasiperiodic HEVPs}
\label{alg: projected spectral method}
\begin{algorithmic}[1]

\Require Quasiperiodic coefficient $\varepsilon(\mathbf r)$, projection matrix
$P$, periodic lifted coefficient $\mathcal E(\mathbf x)$, embedded Bloch wave
vector $\mathbf k$, Fourier number $N$, and physical validation mesh size
$h$.

\Ensure Reconstructed low-dimensional eigenfunction $\widehat u$ and
low-dimensional RQ eigenvalue estimator $\widehat{\lambda}$.

\State Determine the high-dimensional periodic torus $\mathcal T^4$ and the
projection relation according to \eqref{eq: periodic projection variation}: $\mathbf x=P\mathbf r,\,\, \varepsilon(\mathbf r)=\mathcal E(P\mathbf r).$

\State Construct the truncated Fourier index set $\mathcal J$ as in \eqref{set: 4D reciprocal index}.

\State Represent the embedded field in the Fourier form as in \eqref{eq: field fourier span}:
$U(\mathbf x) = \sum_{\boldsymbol{\xi}\in\mathcal J}
U_{\boldsymbol{\xi}}
e^{\imath(\mathbf k+\boldsymbol{\xi})\cdot\mathbf x}$ .

\State Assemble the Fourier-spectral GEVP
derived in \eqref{df: 4D spectral}:
$$
KU=\widetilde{\lambda}MU,
$$
where $ K_{\boldsymbol{\xi}\boldsymbol{\xi}} = (\mathbf k+\boldsymbol{\xi})^{\mathsf T} PP^{\mathsf T} (\mathbf k+\boldsymbol{\xi}), \,\, M_{\boldsymbol{\xi}\boldsymbol{\xi}'} = \mathcal E_{\boldsymbol{\xi}-\boldsymbol{\xi}^{\prime}}.
$

\State Solve the finite-dimensional embedded GEVP to obtain an intermediate eigenpair $(\widetilde{\lambda},U)$.

\State Reconstruct the low-dimensional field by the projection formula
\eqref{sol: 4D projected function}:
$$
\widehat u(\mathbf r)
=
U(P\mathbf r)
=
\sum_{\boldsymbol{\xi}\in\mathcal J}
U_{\boldsymbol{\xi}}
e^{\imath P^{\mathsf T}(\mathbf k+\boldsymbol{\xi})\cdot\mathbf r}.
$$

\State Sample $\widehat u$ on the low-dimensional physical grid with mesh size
$h$, and compute the pointwise RQs using the FD validation
operators:
$
r_s
=
\frac{(\widehat A^{(2)}\widehat{\mathbf u}^{(2)})_s}
{(\widehat B^{(2)}\widehat{\mathbf u}^{(2)}_{\mathrm{o}})_s}.
$

\State Compute the weighted RQ estimator according to \eqref{def: dot rate of 2D}: $\widehat{\lambda}
=
\mathbb E_{\mathrm p}[\{r_s\}]$ .

\State Interpret $(\widehat{\lambda},\widehat u)$ as the projected
low-dimensional eigenpair associated with the original quasiperiodic HEVP.

\end{algorithmic}
\end{algorithm}


\section{Error and Spectrum Analysis for Pointwise  RQs} \label{section: Error analysis}

In this section, we analyze the error and stability of the pointwise RQs associated with the reconstructed vector $\widehat{\mathbf{u}}=[\widehat{u}_{-n},\ldots,\widehat{u}_{n}]^{\mathsf{T}}$, which is obtained from the spectral projection formula \eqref{sol: spectral projection} for the 1D HEVP via the 2D embedded method described in Section~\ref{section: 1Dto2D}. The corresponding analysis for the 2D HEVP obtained from the 4D embedded method in Section~\ref{section: 2Dto4D} can be derived analogously.

For notational simplicity, we suppress the hat notation for the reconstructed solution $\widehat{\boldsymbol{u}}$ in \eqref{sol: spectral projection}  and write $\boldsymbol{u} = [\dots, u_{-1}, u_0, u_1, \dots]^\mathsf{T}$ as an infinite vector. Suppose that $ 0 < \varepsilon_{\min} \le \varepsilon_i \le \varepsilon_{\max},$ for $i = 0, \pm1, \cdots \pm \infty,$ and $\boldsymbol{u}^\mathsf{T} \mathcal{B} \boldsymbol{u} = \sum_{i=0}^{\pm \infty} \varepsilon_i u_i^2 < \infty.$ For the infinite-dimensional discrete problem \eqref{total: discrete form of 2D}, we define the RQ $\rho$ associated with $\boldsymbol{u}$ by

\begin{equation} \label{def: Rayleigh quotient}
\rho = \rho(\boldsymbol{u}) = \frac{\boldsymbol{u}^\mathsf{T} \mathcal{A} \boldsymbol{u}}{\boldsymbol{u}^\mathsf{T} \mathcal{B} \boldsymbol{u}}.
\end{equation}
Based on  \eqref{df: truncted matrices of 2D} and \eqref{df: discretized vector}, for each index satisfying $\varepsilon_i u_i\neq0$, we define the pointwise RQ by
\begin{equation} \label{eq: pointwise ratio}
r_i = \frac{(\widehat{A}\u)_i}{(\widehat{B}\u_{\mathrm{o}})_i}
= \frac{(\mathcal{A}\boldsymbol{u})_i}{(\mathcal{B}\boldsymbol{u})_i},
\quad i = -n, \ldots, n.
\end{equation}

\begin{thm} \label{thm: Hadamard quotient}
    Assume that $\boldsymbol{u}^\mathsf{T} \mathcal{B} \boldsymbol{u} = \sum_{i=0}^{\pm\infty} \varepsilon_i u_i^2 < \infty$, $\varepsilon_i u_i\neq0$, then the RQ $\rho \equiv \rho(\boldsymbol{u})$ can be written as the weighted average of the pointwise RQs $\{r_i\}_{i=0}^{\pm \infty}$ with probability weights

\begin{equation} \label{set: probability density}
 \mathrm{p} = \left\{ p_i = \frac{\varepsilon_{i} u_{i}^2}{\boldsymbol{u}^{\mathsf{T}} \mathcal{B} \boldsymbol{u}} \right\}_{i=0}^{\pm \infty}, \quad \text{i.e.} \quad \rho(\boldsymbol{u}) = \sum_{i=0}^{\pm \infty} p_i r_i = \mathbb{E}_{\mathrm{p}}\left[\{r_i\}_{i=0}^{\pm \infty}\right].
\end{equation}
\end{thm}

\begin{proof}
    Using the definition of $r_i$ in \eqref{eq: pointwise ratio}, we obtain
\begin{align} \label{eq: RQ derive part}
\boldsymbol{u}^\mathsf{T} \mathcal{A} \boldsymbol{u} = \sum_{i=0}^{\pm \infty} u_i (\mathcal{B}\boldsymbol{u})_i \frac{(\mathcal{A}\boldsymbol{u})_i}{(\mathcal{B}\boldsymbol{u})_i} =  \sum_{i=0}^{\pm\infty} u_i (\mathcal{B}\boldsymbol{u})_i r_i.
\end{align}
 Substituting this identity into the definition of $\rho$ in \eqref{def: Rayleigh quotient}, we obtain
\begin{align} \label{eq: RQ derive all}
\rho &=  \frac{\boldsymbol{u}^\mathsf{T} \mathcal{A} \boldsymbol{u}}{\boldsymbol{u}^\mathsf{T} \mathcal{B}\boldsymbol{u}} = \sum_{i=0}^{\pm \infty}
\left(
\frac{u_i (\mathcal{B}\boldsymbol{u})_i}{\boldsymbol{u}^\mathsf{T} \mathcal{B} \boldsymbol{u}}
\right) r_i\notag = \sum_{i=0}^{\pm \infty}
\left(
\frac{\varepsilon_i u_i^2 }{\boldsymbol{u}^\mathsf{T} \mathcal{B}\boldsymbol{u}}
\right) r_i
= \sum_{i=0}^{\pm \infty}
p_i r_i
= \mathbb{E}_{\mathrm{p}}[\{r_i\}_{i=0}^{\pm \infty}],
\end{align}
which proves the claim.
\end{proof}

We next examine the effect of a residual perturbation on the pointwise
quotients. Suppose that $(\lambda, \boldsymbol{u})$ is an approximate eigenpair of the infinite-dimensional discrete problem \eqref{total: discrete form of 2D}. We consider the perturbation equation
\begin{equation} \label{eq: error analysis}
\mathcal{A}\boldsymbol{u} = \lambda \mathcal{B}\boldsymbol{u} + \mathbf{e},
\end{equation}
where $\, \mathbf{e} = \{e_i\}_{i = 0}^{\pm \infty}$ and $e_i$ is the error between $(\mathcal{A}\boldsymbol{u})_i$ and $\lambda (\mathcal{B}\boldsymbol{u})_i$. Then, for each index with $\varepsilon_i u_i\neq0$, \eqref{eq: pointwise ratio} gives
\begin{equation} \label{eq: pointwise error term}
r_i = \lambda + \frac{e_i}{\varepsilon_i u_i},
\quad
\mathbb{E}_{\mathrm{p}}[r_i] = \lambda + \mathbb{E}_{\mathrm{p}}\left[\frac{e_i}{\varepsilon_i u_i}\right].
\end{equation}
This identity shows that the local quotient $r_i$ may become sensitive at points where $|u_i|$ is small, since $\left| r_i-\lambda \right| \leq \frac{|e_i|}{\varepsilon_{\min}|u_i|}$. Thus, small values of $|u_i|$ can amplify the local residual and produce fluctuations in the pointwise quotients. Fortunately, the contribution of such points to the weighted average is moderated by the factor $p_i\propto \varepsilon_i u_i^2$. Therefore, the weighted expectation provides a more stable eigenvalue indicator than the unweighted pointwise quotients, especially when the reconstructed eigenfunction contains small-amplitude regions.

The following Theorem shows that the expectation of the error $\left\{\frac{e_i}{\varepsilon_i u_i}\right\}_{i=0}^{\pm \infty}$ can not be arbitrarily amplified when $|u_i|$ is sufficiently small.

\begin{thm} \label{thm: ratio expectation}
    (i) Assume that $e_{\max} := \max \{|e_i|\}_{i=0}^{\pm \infty} < \infty$ and that there exist constants $C_0>0$, $\alpha>1$, and $N_0\in\mathbb{N}$ such that
    $$|u_i| \le \frac{C_0}{|i|^\alpha}, \; |i|\ge N_0.$$
    Then there exists a constant $C>0$, independent of the residual vector $\mathbf e$, such that
    \begin{equation} \label{neq: expectation estimate}
    \mathbb{E}_{\mathrm{p}}\left[\left\{ \left| \frac{e_i}{\varepsilon_i u_i} \right|\right\}_{i=0}^{\pm \infty}\right]
    \le C e_{\max},
    \end{equation}
    where the probability weights $\mathrm p$ are defined in \eqref{set: probability density}.

    (ii)Consequently,
    \begin{equation} \label{neq: expectation boundary}
    \left| \mathbb{E}_{\mathrm{p}}[r_i] - \lambda \right|
    \le
    \mathbb{E}_{\mathrm{p}}\left[\left| \frac{e_i}{\varepsilon_i u_i} \right|\right]
    \le C e_{\max}.
    \end{equation}

\end{thm}


\begin{proof}
Using the definition of $p_i$, we obtain
\[
\mathbb{E}_{\mathrm{p}}\left[\left| \frac{e_i}{\varepsilon_i u_i} \right|\right]
=
\sum_{i\in\mathbb{Z}}
\frac{\varepsilon_i u_i^2}{\boldsymbol{u}^\mathsf{T}\mathcal{B}\boldsymbol{u}}
\left| \frac{e_i}{\varepsilon_i u_i} \right|
\le
\frac{e_{\max}}{\boldsymbol{u}^\mathsf{T}\mathcal{B}\boldsymbol{u}}
\sum_{i\in\mathbb{Z}} |u_i|.
\]
It therefore suffices to show that $\sum_{i\in\mathbb{Z}} |u_i|<\infty$. This follows from the assumed decay $|u_i|\le C_0 |i|^{-\alpha}$ with $\alpha>1$. Hence
$$
\mathbb{E}_{\mathrm{p}}\left[\left| \frac{e_i}{\varepsilon_i u_i} \right|\right]
\le
C e_{\max}
$$
for some constant $C>0$. The bound \eqref{neq: expectation boundary} then follows from \eqref{eq: pointwise error term} and the triangle inequality.
\end{proof}


We next examine the stability of the pointwise RQ estimator \eqref{eq: pointwise ratio} through the local component ratios of the eigenvector. In particular, the local quotient depends on the neighboring ratios $u_{i+1}/u_i$ and $u_{i-1}/u_i$. From \eqref{eq: pointwise ratio} and Theorem \ref{thm: Hadamard quotient}, we obtain the approximation
\begin{equation} \label{def: ratio estimate}
\rho = \rho(\boldsymbol{u}) \approx r_i
= \frac{(\widehat{A}\u)_i}{\varepsilon_i u_i}
= \frac{1}{h^2}\left(-\frac{1}{\varepsilon_i}\frac{u_{i+1}}{u_i}
+ \frac{2}{\varepsilon_i}
- \frac{1}{\varepsilon_i}\frac{u_{i-1}}{u_i}\right).
\end{equation}
Thus, the behavior of the local ratios provides useful information on the stability of the pointwise RQ estimator. Define
\begin{equation} \label{def: local ratio}
y_i = \frac{u_{i+1}}{u_i},
\quad
\frac{1}{y_{i-1}} = \frac{u_i}{u_{i-1}},
\quad
\boldsymbol{y}_i = \begin{bmatrix}
    u_{i+1} \\ u_i
\end{bmatrix}.
\end{equation}

If the local quotient is close to the RQ $\rho$, then \eqref{def: ratio estimate} leads to the recurrence
\begin{equation} \label{def: iteration}
y_{i} = (2 - h^2\rho \varepsilon_i) - \frac{1}{y_{i-1}}
=: s_i - \frac{1}{y_{i-1}}.
\end{equation}
Equivalently, this scalar recurrence can be written in matrix form as

\begin{equation} \label{def: matrix iteration}
\boldsymbol{y}_i = L_i^{-1} M_i \boldsymbol{y}_{i-1} \quad \text{with} \quad
    M_i = \begin{bmatrix} 1 & 0 \\ 0 & -1 \end{bmatrix}, \quad
    L_i = \begin{bmatrix} -s_i & 1 \\ 1 & 0 \end{bmatrix}.
\end{equation}

To analyze the long-range behavior of this recurrence, we use a symplectic-pair representation of the transfer matrices. We first introduce a very useful lemma for swapping and coalescing two symplectic pairs $(M_i, L_i)$, $(i=1,2)$ of the forms
\begin{equation} \label{def: symplectic pairs}
M_i =
\begin{bmatrix}
A_i & 0 \\
T_i & -I
\end{bmatrix},
\qquad
L_i =
\begin{bmatrix}
- S_i & I \\
A_i^\mathsf{T} & 0
\end{bmatrix}
\quad \text{with } S_i = S_i^\mathsf{T},\; T_i = T_i^\mathsf{T} \in \mathbb{R}^{\ell \times \ell}.
\end{equation}
These pairs satisfy
\[
M_i J M_i^\mathsf{T} = L_i J L_i^\mathsf{T},
\quad \text{with }
J =
\begin{bmatrix}
0 & I \\
- I & 0
\end{bmatrix},
\]
and eigenvalues of $(M_i,L_i)$ form reciprocal pairs.

\begin{lem} \label{lemma: symplectic pairs}
Let $(M_i,L_i)$, $i=1,2$, be two symplectic pairs of the form \eqref{def: symplectic pairs}. Suppose that $\left(T_2 - S_1\right)^{-1}$ exists. Then a swapping and coalescing symplectic pair for $\{(M_i, L_i)\}_{i=1}^2$ can be computed by
\begin{subequations} \label{total: lemma symplectic}
    \begin{align}
(L_2^{-1} M_2)(L_1^{-1} M_1)
&= L_1(\widehat{L}_1^{-1} \widehat{M}_2 ) M_1
= (\widehat{L}_1 L_1)^{-1} (\widehat{M}_2 M_1) \nonumber \\
&=\begin{bmatrix}
- S_{(12)} & I \\
A_{(12)}^\mathsf{T} & 0
\end{bmatrix} ^{-1}
\begin{bmatrix}
A_{(12)} & 0 \\
T_{(12)} & -I
\end{bmatrix},
\end{align}
where
\begin{equation}
\widehat{L}_1 =
\begin{bmatrix}
I & - A_2 (T_2 - S_1)^{-1} \\
0 & A_1^\mathsf{T} (T_2 - S_1)^{-1}
\end{bmatrix},
\qquad
\widehat{M}_2 =
\begin{bmatrix}
A_2 (T_2 - S_1)^{-1} & 0 \\
- A_1^\mathsf{T} (T_2 - S_1)^{-1} & I
\end{bmatrix},
\end{equation}
and
\begin{align}
T_{(12)} &= T_1 - A_1^\mathsf{T} (T_2 - S_1)^{-1} A_1, \\
S_{(12)} &= S_2 + A_2 (T_2 - S_1)^{-1} A_2^\mathsf{T}, \\
A_{(12)} &= A_2 (T_2 - S_1)^{-1} A_1.
\end{align}
\end{subequations}
\end{lem}

\begin{proof}
    In light of the proof of Thm.~2.2 in \cite{lin2006convergence} which performs a doubling transformation to merge the square of the symplectic matrix $ (L_1^{-1} M_1)^2 = \widehat{L}^{-1}\widehat{M},$ we can also perform the swapping and coalescing steps as in \eqref{total: lemma symplectic} to get
$$
L_{(12)} =
\begin{bmatrix}
- S_{(12)} & I \\
A_{(12)}^{\mathsf{T}} & 0
\end{bmatrix},
\qquad
M_{(12)} =
\begin{bmatrix}
A_{(12)} & 0 \\
T_{(12)} & -I
\end{bmatrix}.
$$
\end{proof}

W.L.O.G., the matrix recurrence in \eqref{def: matrix iteration} can be written as the product of the symplectic matrices
$\{L_i^{-1} M_i\}_{i=1}^n$, i.e.,
\begin{equation} \label{eq: matrix recurrence}
\boldsymbol{y}_{n+1}
=
(L_n^{-1}M_n)(L_{n-1}^{-1}M_{n-1})\cdots(L_1^{-1}M_1)\,\boldsymbol{y}_1.
\end{equation}
From Lemma \ref{lemma: symplectic pairs} it follows that \eqref{eq: matrix recurrence} can be swapped and colleasced as a symplectic pair as
\begin{align} \label{dv: symplectic pair swap and collapse}
    \boldsymbol{y}_{n+1} &=
(L_n^{-1}M_n)\cdots(L_{(3,4)}^{-1}M_{(3,4)})(L_{(1,2)}^{-1}M_{(1,2)})\,\boldsymbol{y}_1
=
\big(L_{(12\cdots n)}^{-1} M_{(1 2\cdots n)}\big)\,\boldsymbol{y}_1 \notag\\
&= \begin{pmatrix}
- s_{(1 \cdots n)} & 1 \\
a_{(1 \cdots n)} & 0
\end{pmatrix}^{-1}
\begin{pmatrix}
a_{(1\cdots n)} & 0 \\
t_{(1\cdots n)} & -1
\end{pmatrix}
\,\boldsymbol{y}_1
 \equiv
\begin{pmatrix}
- \widetilde{s}_n & 1\\
\widetilde{a}_n & 0
\end{pmatrix}^{-1}
\begin{pmatrix}
\widetilde{a}_n & 0 \\
\widetilde{t}_n & -1
\end{pmatrix}
\,\boldsymbol{y}_1 \notag \\
& \equiv \widetilde{L}_n^{-1}\widetilde{M}_n\boldsymbol{y}_1.
\end{align}


\begin{thm} \label{thm: expectation of ratio}
    Suppose that $ \boldsymbol{y}_n = [\,u_{n+1},\,u_n\,]^{\mathsf{T}}$ is close to the eigenvector of $\widetilde{L}_n^{-1}\widetilde{M}_n$ corresponding to the smallest eigenvalue in module.  Assume that $\widetilde{a}_n^2 < \eta^2$ is small, $|\widetilde{s}_n|$ and $|\widetilde{t}_n|$ in \eqref{dv: symplectic pair swap and collapse} have the same order. Then the ratio $ \frac{1}{y_n} = \frac{u_n}{u_{n+1}} $ in \eqref{def: local ratio} is of order $\mathcal{O}(1)$.
\end{thm}

\begin{proof}
Since $\widetilde{a}_n \approx \mathcal{O}(\eta)$ is small, $|\widetilde{s}_n|$ and $|\widetilde{t}_n|$ have the same order,
we consider the following three cases.

\begin{itemize}
    \item[(i)]If $|\widetilde{s}_n| = |\widetilde{t}_n| = \mathcal{O}(1)$, then it holds that
    \begin{equation} \label{eq: order approx}
    \left(
    \begin{bmatrix}
    \widetilde{a}_n & 0 \\
    \widetilde{t}_n & -1
    \end{bmatrix},
    \begin{bmatrix}
    -\widetilde{s}_n & 1 \\
    \widetilde{a}_n & 0
    \end{bmatrix}
    \right)
    =
    \left(
    \begin{bmatrix}
    \mathcal{O}(\eta) & 0 \\
    \mathcal{O}(1) & -1
    \end{bmatrix},
    \begin{bmatrix}
    \mathcal{O}(1) & 1 \\
    \mathcal{O}(\eta) & 0
    \end{bmatrix}
    \right).
    \end{equation}

    From \eqref{eq: order approx} we see that$
    \begin{bmatrix}
\mathcal{O}(\eta) & 0 \\
\widetilde{t}_n & -1
\end{bmatrix}
\begin{bmatrix}
1 \\
\widetilde{t}_n
\end{bmatrix}
\approx \mathcal{O}(\eta).
$ Hence the vector $(1, \widetilde{t}_n)^{\mathsf{T}}$ is an approximate eigenvector
i.e.,$\, \boldsymbol{y}_n = (u_{n+1},u_n)^{\mathsf{T}} \approx (1,\widetilde{t}_n)^{\mathsf{T}} $
is close to the eigenvector of $(\widetilde{M}_n,\widetilde{L}_n)$ of \eqref{dv: symplectic pair swap and collapse} corresponding  to the smallest eigenvalue of order $\mathcal{O}(\eta)$.
Thus,$ \frac{u_n}{u_{n+1}} \approx \widetilde{t}_n = \mathcal{O}(1).$ More precisely, using the explicit representation $\boldsymbol{y}_{n+1}=\widetilde{L}_n^{-1}\widetilde{M}_n\boldsymbol{y}_1$, we obtain
$$
\begin{bmatrix}
u_{n+1} \\
u_n
\end{bmatrix}
=
\begin{bmatrix}
-\widetilde{s}_n & 1 \\
\widetilde{a}_n & 0
\end{bmatrix}^{-1}
\begin{bmatrix}
\widetilde{a}_n & 0 \\
\widetilde{t}_n & -1
\end{bmatrix}
\begin{bmatrix}
u_1 \\
1
\end{bmatrix}
=
\begin{bmatrix}
\widetilde{a}_n^{-1}\widetilde{t}_n & -\widetilde{a}_n^{-1} \\
\widetilde{a}_n +  \widetilde{a}_n^{-1}\widetilde{s}_n\widetilde{t}_n& -\widetilde{s}_n \widetilde{a}_n^{-1}
\end{bmatrix}
\begin{bmatrix}
u_1 \\
1
\end{bmatrix}.
$$

Therefore, we have
\begin{equation} \label{eq: ratio order}
\frac{u_n}{u_{n+1}}
=
\frac{\widetilde{s}_n(\widetilde{t}_n u_1 -1) + \widetilde{a}_n^2 u_1}
{\widetilde{t}_n u_1 -1}
=
\widetilde{s}_n + \frac{\widetilde{a}_n^2 u_1}{\widetilde{t}_n u_1 -1}
= \mathcal{O}(1),
\end{equation}
provided that $|\widetilde{t}_n u_1-1|$ is bounded away from zero, then the second term is
$\mathcal{O}(\eta^2)$.

\item[(ii)] If $|\widetilde{s}_n| = |\widetilde{t}_n| = \mathcal{O}(\frac{1}{\eta})$, then we have the equivalence transformation
\begin{align} \label{def: equiv symplectic}
&\left(
\begin{bmatrix}
\widetilde{a}_n & 0 \\
\widetilde{t}_n & -1
\end{bmatrix},
\begin{bmatrix}
-\widetilde{s}_n & 1 \\
\widetilde{a}_n & 0
\end{bmatrix}
\right)
\overset{eq.}{\sim}
\left(
\begin{bmatrix}
\mathcal{O}(\eta^2) & 0 \\
\mathcal{O}(1) & -1
\end{bmatrix},
\begin{bmatrix}
\mathcal{O}(1) & 1 \\
\mathcal{O}(\eta^2) & 0
\end{bmatrix}
\right) \notag \\
= &
\left(
\begin{bmatrix}
\widehat{a}_n & 0 \\
\widehat{t}_n & -1
\end{bmatrix},
\begin{bmatrix}
-\widehat{s}_n & 1 \\
\widehat{a}_n & 0
\end{bmatrix}
\right) \equiv (\widehat{M}_n,\widehat{L}_n) .
\end{align}
From \eqref{def: equiv symplectic} we see that$\begin{bmatrix}
\mathcal{O}(\eta^2) & 0 \\
\widehat{t}_n & -1
\end{bmatrix}
\begin{bmatrix}
1 \\
\widehat{t}_n
\end{bmatrix}
\approx \mathcal{O}(\eta^2),
$ i.e. $ \boldsymbol{y}_n = (u_{n+1},u_n)^{\mathsf{T}} \approx (1,\widehat{t}_n)^{\mathsf{T}}$ is close to the eigenvector of $(\widehat{M}_n,\widehat{L}_n)$ of \eqref{def: equiv symplectic} corresponding to the smallest eigenvalue of order $\mathcal{O}(\eta^2)$. Thus,$ \frac{u_n}{u_{n+1}} \approx \widehat{t}_n = \mathcal{O}(1).$ From \eqref{eq: ratio order} it also holds that
\begin{equation} \label{eq: proof of ratio order}
\frac{u_n}{u_{n+1}}
=
\widehat{s}_n + \frac{\widehat{a}_n^2 u_1}{\widehat{t}_n u_1 -1}
= \mathcal{O}(1).
\end{equation}
Again assuming that the denominator is bounded away from zero.

\item[(iii)] If $|\widetilde{s}_n| = |\widetilde{t}_n| = \mathcal{O}(\eta)$, then we have the equivalence transformation
\begin{align} \label{dv: symplectic derive}
\left(
\begin{bmatrix}
\widetilde{a}_n & 0 \\
\widetilde{t}_n & -1
\end{bmatrix},
\begin{bmatrix}
-\widetilde{s}_n & 1 \\
\widetilde{a}_n & 0
\end{bmatrix}
\right)
&\overset{eq.}{\sim}
\left(
\begin{bmatrix}
\mathcal{O}(1) & 0 \\
\mathcal{O}(1) & -1
\end{bmatrix},
\begin{bmatrix}
\mathcal{O}(1) & 1 \\
\mathcal{O}(1) & 0
\end{bmatrix}
\right) \notag \\
=
\left(
\begin{bmatrix}
\widehat{a}_n & 0 \\
\widehat{t}_n & -1
\end{bmatrix},
\begin{bmatrix}
-\widehat{s}_n & 1 \\
\widehat{a}_n & 0
\end{bmatrix}
\right)
&\equiv
(\widehat{M}_n,\widehat{L}_n).
\end{align}
Let $(z_{n+1},z_n)^{\mathsf{T}}$ be the eigenvector of $\widehat{L}_n^{-1}\widehat{M}_n$ corresponding to the eigenvalue $\widehat{\lambda}_n$ satisfying $|\widehat{\lambda}_n|<1$. Then from \eqref{dv: symplectic derive} it holds, in general,  that
\begin{equation} \label{dv: iteration derive}
\widehat{L}_n^{-1}\widehat{M}_n
\begin{bmatrix}
z_{n+1} \\
z_n
\end{bmatrix}
=
\widehat{\lambda}_n
\begin{bmatrix}
z_{n+1} \\
z_n
\end{bmatrix} =
\widehat{\lambda}_n
\begin{bmatrix}
\mathcal{O}(1) \\
\mathcal{O}(1)
\end{bmatrix}.
\end{equation}
Thus, from \eqref{eq: ratio order} by replacing “tilde” by “hat” and $|\widehat{a}_n| = |\widehat{s}_n|=|\widehat{t}_n|=\mathcal{O}(1)$ in \eqref{dv: symplectic derive} follows that
\begin{equation} \label{eq: order of ratio result}
\frac{u_n}{u_{n+1}}
=
\widehat{s}_n + \frac{\widehat{a}_n^2 u_1}{\widehat{t}_n u_1 -1}
= \mathcal{O}(1).
\end{equation}
\end{itemize}
Combining the three cases proves the theorem.
\end{proof}

\begin{remark}

\begin{itemize}
    \item[(i)] From Theorem \ref{thm: Hadamard quotient}, we see that the expectation of the pointwise RQs\\ $\mathbb{E}_{\widehat{\mathrm{p}}}\left[\{r_i\}_{i=0}^{\pm n}\right]$ for $\{r_i\}_{i=0}^{\pm \infty}$ in \eqref{eq: pointwise ratio} with respect to the probability $\widehat{\mathrm{p}} = \left\{\widehat{p}_i = \frac{\varepsilon_i u_i^2}{\u_{\mathbf{o}}^\mathsf{T} \widehat{B}\u_{\mathbf{o}}}\right\}_{i=0}^{\pm n}$ converges to $\mathbb{E}_{\mathrm{p}}\left[\{r_i\}_{i=0}^{\pm \infty}\right] = \rho(\boldsymbol{u}).$ This establishes a consistency property of the estimator: as the sampling of $u_i$ becomes sufficiently dense and the sampled sequence $\{u_i\}_{i=0}^{\pm n}$ provides an increasingly accurate approximation to the true eigenfunction $\boldsymbol{u}$, the cropped expectation $\mathbb{E}_{\mathrm{p}}\!\left[\{r_i\}_{i=0}^{\pm n}\right]$ converges to $\rho(\boldsymbol{u})$ as $n\to\infty$.

    Since $\sum_{i=0}^{\pm \infty}\varepsilon_i u_i^2<\infty$, we have $|u_i|\to0$ as $|i|\to\infty$. In practice, we do not need to expand the index set for $i$ too large. More precisely, we use a pointwise RQ \eqref{eq: pointwise ratio} to approach the global averaged estimator.

    \item [(ii)]

    In \eqref{eq: error analysis}, the residual component $e_i$ measures the discrepancy between $(\mathcal{A}\boldsymbol{u})_i$ and  $\lambda(\mathcal{B}\boldsymbol{u})_i$. From  \eqref{eq: pointwise error term}, the deviation of $\mathbb{E}_{\mathrm p}[r_i]$ from $\lambda$ is governed by $\mathbb{E}_{\mathrm p} \left[ \left|\frac{e_i}{\varepsilon_i u_i} \right| \right].$
    Although the local factor $1/u_i$ may amplify the pointwise error when $|u_i|$ is small, Theorem~\ref{thm: ratio expectation} shows that, under the decay condition
    $ |u_i|\le C_0 |i|^{-\alpha},
    \,\,\, \alpha>1, $
    the weighted error satisfies $
    \mathbb{E}_{\mathrm p} \left[ \left|
    \frac{e_i}{\varepsilon_i u_i}
    \right| \right] \le C e_{\max}.
    $
    Thus, the weighted expectation is not arbitrarily amplified by small local
    values of $|u_i|$.

    \item [(iii)]
    Items (i) and (ii) justify $\mathbb{E}_{\mathrm p}[r_i]$ as a stable estimator of $\rho(\boldsymbol{u})$, and hence of the target eigenvalue when $\rho(\boldsymbol{u})\approx\lambda$. In addition, the expression \eqref{def: ratio estimate} shows that the pointwise RQ depends on the local ratios $\frac{u_{i\pm1}}{u_i}$. Theorem~\ref{thm: expectation of ratio} provides a stability criterion for these ratios under the assumptions $\widetilde a_n=\mathcal{O}(\eta)$ and $|\widetilde s_n|$ and $|\widetilde t_n|$ being of the same order. This supports the stability of the pointwise RQ approach under the corresponding coalesced recurrence structure.
\end{itemize}

\end{remark}

The same construction applies to the 2D HEVP obtained from the 4D embedded formulation in Section~\ref{section: 2Dto4D}. From \eqref{df: 2D TMmode}--\eqref{total: discrete infinite 2D} and \eqref{total: finite matrix derive 4D}--\eqref{def: dot rate of 2D}, we define the RQ

\begin{equation} \label{def: RQ of 2Dto4D}
\rho = \rho(\boldsymbol{u}^{(2)}) =
\frac{ \boldsymbol{u}^{(2)}\!^\mathsf{T} \mathcal{A}^{(2)} \boldsymbol{u}^{(2)} }
{ \boldsymbol{u}^{(2)}\!^\mathsf{T} \mathcal{B}^{(2)} \boldsymbol{u}^{(2)} }, \quad \text{with} \quad
\boldsymbol{u}^{(2)}\!^\mathsf{T} \mathcal{B}^{(2)} \boldsymbol{u}^{(2)} =
\sum_{s\in \mathcal{S}^{(2)}(\infty)} \varepsilon_{s} u_{s}^2,
\end{equation}
the pointwise RQs and the corresponding eigenvalue estimator are given by
\begin{equation} \label{def: statistic r_ij}
r_{s} =
\frac{\left(\widehat{A}^{(2)} \u^{(2)}\right)_{s}}{\left(\widehat{B}^{(2)} \widehat{\u}_{\mathrm{o}}^{(2)}\right)_{s} } = \frac{\left(\widehat{A}^{(2)} \u^{(2)}\right)_{s}}{\varepsilon_{s} u_{s} }
\quad \text{and} \quad
\widehat{\lambda}(k) = \mathbb{E}_\mathrm{p}\left[ \{ r_{s} \}_{s\in \mathcal{S}^{(2)}(n)} \right],
\end{equation}
where $ \mathrm{p} = \left\{ p_{s} =
\frac{\varepsilon_{s} u_{s}^2}{\u_{\mathrm{o}}^{(2)\mathsf{T}} \widehat{B}^{(2)} \u^{(2)}_{\mathrm{o}}}
\right\}_{s\in \mathcal{S}^{(2)}(n)}$ and $\mathcal{S}^{(2)}(n) = \text{vec}(\mathcal{S}(n) \times \mathcal{S}(n))$ with $\mathcal{S}(n) = \{-n, \ldots, 0, \ldots, n\}$ $(n =$ finite or $\infty).$

Then the following results follow from Theorem \ref{thm: Hadamard quotient} and \ref{thm: ratio expectation} immediately.

\begin{thm} \label{thm: 2Dto4D Rayleigh quotient}
Let $\boldsymbol{u}^{(2)}$ be the vectorized 2D sequence and assume that $(\boldsymbol{u}^{(2)})^{\mathsf T} \mathcal{B}^{(2)} \boldsymbol{u}^{(2)}
=
\sum_{s\in\mathcal{S}^{(2)}(\infty)} \varepsilon_s u_s^2 <\infty .$
Then the following statements hold.
    \begin{itemize}
\item[(i)]  The RQ $\rho(\boldsymbol{u}^{(2)})$ defined in \eqref{def: RQ of 2Dto4D} can be represented as the weighted expectation of the pointwise RQs
$ \rho = \rho(\boldsymbol{u}^{(2)}) = \mathbb{E}_\mathrm{p}\Big[ \{ r_{s} \}_{s\in\mathcal{S}^{(2)}(\infty)} \Big]$
with $ \mathrm{p} = \left\{ p_{s} =
\frac{\varepsilon_{s} u_{s}^2}{\boldsymbol{u}^{(2)}\!^\mathsf{T} \mathcal{B}^{(2)} \boldsymbol{u}^{(2)}}
\right\}_{s\in\mathcal{S}^{(2)}(\infty)}. $

\item[(ii)] Let $(\lambda, \boldsymbol{u}^{(2)})$ be an approximate eigenpair of \eqref{df: 2D TMmode} and define the residual component by $e_{s} = (\mathcal{A}^{(2)} \boldsymbol{u}^{(2)})_{s} - \lambda (\mathcal{B}^{(2)} \boldsymbol{u}^{(2)})_{s}.$ Let $\, e_{\max} = \max \{|e_{s}|\}_{s\in\mathcal{S}^{(2)}(\infty)}$, assume additionally that $ \sum_{s\in\mathcal{S}^{(2)}(\infty)} |u_{s}| < C_0. $ Then it holds that $\,\, \mathbb{E}_\mathrm{p}\!\left[ \Big\{ \left| \frac{e_{s}}{\varepsilon_{s} u_{s}} \right| \Big\}_{s\in\mathcal{S}^{(2)}(\infty)} \right]
\le C\, e_{\max}, $
where $ C = \frac{C_0}{\boldsymbol{u}^{(2)^\mathsf{T}} \mathcal{B}^{(2)} \boldsymbol{u}^{(2)}}.$
\item[(iii)] The following bias estimate holds:
$
\left| \mathbb{E}_\mathrm{p}[r_{s}] - \lambda \right|
\le
\mathbb{E}_\mathrm{p}\!\left[ \left| \frac{e_{s}}{\varepsilon_{s} u_{s}} \right| \right]
\le C e_{\max}. $
\end{itemize}
\end{thm}

For a finite index set  $(i,j) \in \mathcal{S}(n) \times \mathcal{S}(n)$, the corresponding pointwise RQ is given by
\begin{equation} \label{eq: approximate of rho}
r_{ij}
=
\frac{(\widehat{A}^{(2)}\u^{(2)})_{ij}}{\varepsilon_{ij} u_{ij}}
=
\frac{(\mathcal{A}^{(2)} \boldsymbol{u}^{(2)})_{ij}}{(\mathcal{B}^{(2)} \boldsymbol{u}^{(2)})_{ij}}.
\end{equation}
By Theorem~\ref{thm: 2Dto4D Rayleigh quotient}, the global RQ is the weighted average of such local quantities. Thus, the local values $\{r_{ij}\}$ provide pointwise approximation on the global estimator $\rho(\boldsymbol{u}^{(2)})$.
From \eqref{df: 2D TMmode} and \eqref{df: matrices of 4D} for convenience, the five-point FD discretization of \eqref{evp: 2D TM mode} can be written as
\begin{equation} \label{eq: five point discretization}
2s_{ij} u_{ij}
\equiv
(4 - h^2\rho \varepsilon_{ij}) u_{ij}
=
u_{i+1,j} + u_{i-1,j} + u_{i,j+1} + u_{i,j-1},
\end{equation}
where $ s_{ij} = \frac{1}{2}\big(4 - h^2\rho \varepsilon_{ij}\big).$
Let
\begin{equation} \label{def: y_ij and z_ij}
y_{ij} = \frac{u_{i+1,j}}{u_{ij}} \,\,\,
\text{and} \,\,\,
z_{ij} = \frac{u_{i,j+1}}{u_{ij}}
\end{equation}
denote the ratios between $u_{i+1,j}$ and $u_{i,j}$ rowwisely, as well as, $u_{i,j+1}$ and $u_{i,j}$ columnwisely. Then \eqref{eq: five point discretization} can be separated into two 1D ratio recurrences,
\begin{equation} \label{def: p_ij with y and z}
s_{ij} = y_{ij} + \frac{1}{y_{i-1,j}},
\qquad
s_{ij} = z_{ij} + \frac{1}{z_{i,j-1}}.
\end{equation}
Using the same symplectic-pair argument as in Theorem~\ref{thm: expectation of ratio}, we obtain the following stability criterion for the row-wise and column-wise local ratios.

\begin{thm} \label{thm: estimation of rho of 4D}
    Suppose that $ \boldsymbol{y}_{n,j} = (u_{n+1,j}, u_{n,j})^\mathsf{T}, $ and $
\boldsymbol{z}_{i,n} = (u_{i,n+1}, u_{i,n})^\mathsf{T} $
are close to the eigenvectors of $ \, \widetilde{L}_{n,j}^{-1}\widetilde{M}_{n,j}$
and $\, \widetilde{L}_{i,n+1}^{-1}\widetilde{M}_{i,n}$ corresponding to the
smallest eigenvalues, respectively, where
\begin{subequations} \label{total: symplectic matrices}
    \begin{equation} \label{def: symplectic 1}
\widetilde{M}_{n,j} =
\begin{bmatrix}
\widetilde{a}_{n,j} & 0 \\
\widetilde{t}_{n,j} & -1
\end{bmatrix},
\qquad
\widetilde{L}_{n,j} =
\begin{bmatrix}
-\widetilde{s}_{n,j} & 1 \\
\widetilde{a}_{n,j} & 0
\end{bmatrix},
\end{equation}

\begin{equation}\label{def: symplectic 2}
\widetilde{M}_{i,n} =
\begin{bmatrix}
\widetilde{a}_{i,n} & 0 \\
\widetilde{t}_{i,n} & -1
\end{bmatrix},
\qquad
\widetilde{L}_{i,n} =
\begin{bmatrix}
-\widetilde{s}_{i,n} & 1 \\
\widetilde{a}_{i,n} & 0
\end{bmatrix}.
\end{equation}
\end{subequations}
Assume that $ \widetilde{a}_{n,j}^2 < \eta^2 $ and $ \widetilde{a}_{i,n}^2 < \eta^2, |\widetilde{s}_{n,j}|$ and $|\widetilde{t}_{n,j}|$, as well as,  $|\widetilde{s}_{i,n}|$ and $|\widetilde{t}_{i,n}|$ have the same orders, respectively. Then the ratios $ \frac{1}{y_{n,j}} = \frac{u_{n,j}}{u_{n+1,j}}, $ and $ \frac{1}{z_{i,n}} = \frac{u_{i,n}}{u_{i,n+1}} $ in \eqref{def: y_ij and z_ij} are of order $\mathcal{O}(1)$.
\end{thm}

We finally discuss the relation between the spectrum of the original low-dimensional HEVP and that of its embedded high-dimensional formulation. We present the argument for the 1D HEVP and its associated 2D embedded EVP; the 2D-to-4D case can be treated analogously.

Using \eqref{evp: 1D HEVPs}, \eqref{eq: projection relationship}--\eqref{eq: projection field} and the chain rule calculation on \eqref{evp: 1Dto2D continous projection} and \eqref{eq: Bloch expand}, we have
\begin{equation}\label{eq: fprojection operator}
\varepsilon(r)^{-1} \nabla_r^2 u(r)
=
\mathcal{E}(Pr)^{-1} \left( \nabla_P + i P^\mathsf{T} \k \right)^2 U_p(Pr),
\end{equation}
where $\k$ is a given Bloch vector. From \eqref{df: Spectral Method 2D}, we define the infinite-dimensional Floquet-Bloch operators
\begin{equation}\label{def: infinite martrix}
\mathcal{K}_\infty(\k) = \bigl[ \mathcal{K}_{\boldsymbol{\xi}\boldsymbol{\xi}}(\k) \bigr], \quad \mathcal{M}_\infty = \bigl[ \mathcal{M}_{\boldsymbol{\xi}\boldsymbol{\xi}^{\prime}} \bigr] = \bigl[ \mathcal{E}_{\boldsymbol{\xi}^{\prime}-\boldsymbol{\xi}} \bigr],
\end{equation}
where \(\boldsymbol{\xi}, \boldsymbol{\xi}' \in \mathcal{J}_N\), and the truncation size \(N \to \infty\) as in \eqref{set: reciprocal space}. Then applying cut-and-project technique to \eqref{total: discrete form of 2D}, \eqref{df: Spectral Method 2D} and using \eqref{eq: fprojection operator} we obtain formally
\begin{equation} \label{eq: finite estimate to infinite}
\mathcal{B}^{-1} \mathcal{A} \boldsymbol{u} + \mathcal{O}(h^2)
=
\left( \mathcal{M}_\infty^{-1} \mathcal{K}_\infty \right) U
=
\widetilde{\lambda}\, U, \quad h \to 0^+.
\end{equation}


The reconstructed 1D eigenfunction $\widehat{u}(r)=U(Pr)$ is obtained by restricting the embedded eigenfunction to the irrational slice $\mathbf{x}=Pr$. This slice is dense in the torus $\mathcal{T}^2$, but it is not, in general, an invariant subspace of the operator pair $(\mathcal{K}_\infty(\mathbf{k}),\mathcal{M}_\infty)$ in \eqref{def: infinite martrix}. Therefore, the embedded eigenvalue $\widetilde{\lambda}$ should not be identified directly with an eigenvalue of the original infinite-dimensional low-dimensional problem. Instead, the projected RQ provides the appropriate low-dimensional spectral indicator.

\begin{thm}
Suppose $ \widetilde{\lambda} \in  \sigma\bigl(\mathcal{K}_\infty(\k), \mathcal{M}_\infty\bigr)$ with $U$ being the associate eigenvector in \eqref{eq: finite estimate to infinite}.  Let $\boldsymbol{u}=U(Pr)$ be the corresponding reconstructed low-dimensional sequence. Then the RQ
$\rho(\boldsymbol{u})$ satisfies $
\rho(\boldsymbol{u})+\mathcal{O}(h^2) \in
\sigma(\mathcal{A},\mathcal{B}), $ as $ h\to0^+, $
in the sense of an approximate spectral value of the low-dimensional
discretized operator pair.
\end{thm}

\begin{proof}
 For a $ \widetilde{\lambda} \in  \sigma\bigl(\mathcal{K}_\infty(\k), \mathcal{M}_\infty\bigr)$, \eqref{eq: finite estimate to infinite} implies that $\mathcal{A}\boldsymbol{u} + \mathcal{O}(h^2)$ is approximately parallel to $\mathcal{B} \boldsymbol{u}$.
Hence $\boldsymbol{u}$ is an approximate eigenvector of $(\mathcal{A},\mathcal{B})$ with a residual of order \(\mathcal{O}(h^2)\). Therefore, there exists an approximate spectral value $\lambda \in \sigma(\mathcal{A}, \mathcal{B})$ such that $\mathcal{A}\boldsymbol{u} = \lambda \mathcal{B} \boldsymbol{u} + \mathcal{O}(h^2).$ By Theorem \ref{thm: Hadamard quotient} $\rho(\boldsymbol{u}) = \mathbb{E}_{\mathrm{p}} \left[
\left\{
\frac{(\mathcal{A}\boldsymbol{u})_i}{(\mathcal{B}\boldsymbol{u})_i}
\right\}_{i=0}^{\pm \infty}
\right]$. The weighted-error estimate in Theorem~\ref{thm: ratio expectation} shows that
the residual of order $\mathcal{O}(h^2)$ is not arbitrarily amplified by the
componentwise division. Consequently, $\rho(\boldsymbol{u}) + \mathcal{O}(h^2) = \lambda \in \sigma(\mathcal{A}, \mathcal{B}), \text{as}  \,\, h \to 0^+. $
\end{proof}

We further record the corresponding spectral inclusion relation for the
continuous operators.

\begin{thm} \label{thm: inclusion relation}
It holds that
\begin{equation}\label{set: spectrum including}
\sigma\bigl(\nabla_r^2,\varepsilon(r)\bigr)
\subseteq
\overline{
\bigcup_{\k \in \mathbf{B}}
\sigma\Bigl( (\nabla_P + \imath P^\mathsf{T} \k)^2, \mathcal{E}(\x) \Bigr)
},
\end{equation}
where \(\mathbf{B}\) is the Brillouin zone of the Bloch vector $\k$ in \(\mathcal{T}^2\).
\end{thm}

\begin{proof}
For a $\lambda \in \sigma(\nabla_r^2,\varepsilon(r))$, let $u(r)$ be the associated eigenfunction.  Since $u(r)$ and $\varepsilon(r)$ is quasiperiodic, there are periodic functions $U$ and $\mathcal{E}$ on \(\mathcal{T}^2\) such that $ u(r) = U(Pr),$ and $\varepsilon(r) = \mathcal{E}(Pr)$, i.e., we have
\begin{equation}\label{eq: infinite projection}
\varepsilon(r)^{-1} \nabla_r^2 u(r)
=
\mathcal{E}(\x)^{-1} \nabla_P^2 U(\x)\big|_{\x=Pr}.
\end{equation}
However, $U(\x)$ and $\mathcal{E}(\x)$ are not Fourier finite. Let $P_N$ denote the Fourier projection onto the modes with truncation size $N$, and define $U_N = P_N U  \,\, \mathcal{E}_N = P_N \mathcal{E}$. As $N\to\infty$, the Fourier truncations approximate the corresponding periodic functions in the relevant norm. Therefore,
\begin{equation}\label{dv: spectral trunc est}
\bigl\| (\nabla_P^2 - \lambda \mathcal{E}_N) U_N \bigr\| \leq \left\|\left(
\nabla_P^2 - \left(\nabla_P^{(N)}\right)^2
\right)U_N\right\| +
\left\|\left(
\left(\nabla_P^{(N)}\right)^2 - \lambda \mathcal{E}_N
\right)U_N\right\|
\to 0,
\quad \text{as } N \to \infty,
\end{equation}
where $\left(\nabla_P^{(N)}\right)^2$ denotes the $N$-dimensional truncation of $\nabla_P^2$. So, $\lambda \in \sigma(\nabla_P^2,\mathcal{E})$. Since $ \sigma(\nabla_P^2,\mathcal{E}) = \bigcup_{\k \in \mathbf{B}} \sigma\Bigl( (\nabla_P + \imath P^\mathsf{T} k)^2, \mathcal{E} \Bigr)\equiv \Sigma, $ it follows that $\lambda$ is in the closure of $\Sigma$.
\end{proof}

\section{Numerical Experiments}

This section presents numerical experiments to illustrate the main components of the proposed framework, including quasiperiodic reconstruction, pointwise RQ validation, smoothing treatment for discontinuous media, and band structure comparison with supercell approximations.
All numerical experiments in this section are carried out in MATLAB R2024a on a desktop computer equipped with an AMD Ryzen AI 7 H 350 processor with Radeon 860M at 2.00 GHz and 32 GB RAM. The GEVPs arising from the projected Fourier-pseudospectral discretization are solved by the built-in \texttt{eigs} routine.

\subsubsection*{Example 1 (Quasiperiodic structure of the reconstructed eigenfunctions)}

This example illustrates how the eigenfunctions computed from the embedded periodic formulation are projected back to the original low-dimensional physical domain. The purpose is to show that, after projection, the reconstructed functions inherit the quasiperiodic structure induced by the irrational projection relation.

We first consider the 1D-to-2D case. The quasiperiodic coefficient is given by
\begin{equation}
\varepsilon(r) = \cos(p_1r)+\cos(p_2r)+5,
\qquad
p_1=1,\quad p_2=\frac{\sqrt{5}-1}{2}.
\label{def:1D_quasi_evp_func}
\end{equation}
Equivalently, $\varepsilon(r)$ can be viewed as the restriction of the 2D periodic function $\mathcal E(x_1,x_2)=\cos x_1+\cos x_2+5$ along the irrational line $(x_1,x_2)=(p_1r,p_2r)$. Since the slope $p_2/p_1$ is irrational, this line winds densely on the 2D torus $\mathcal T^2$, and the restricted function is quasiperiodic rather than periodic in $r$. This projection mechanism is illustrated in Figure~\ref{fig:1Dto2D_coefficient_projection}.

\begin{figure}[H]
\centering
\begin{subfigure}[t]{0.48\textwidth}
    \centering
    \includegraphics[width=\textwidth]{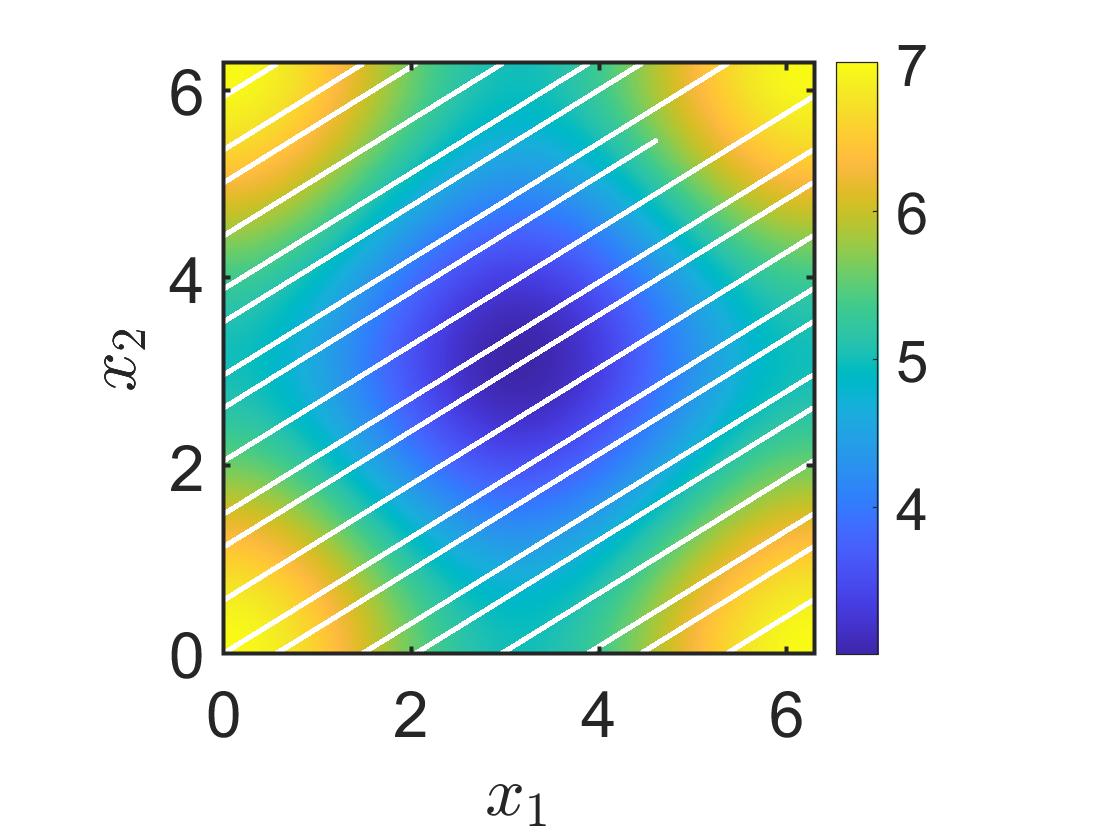}
    \caption{Periodic coefficient on $\mathcal T^2$.}
    \label{subfig:2D_periodic_projection}
\end{subfigure}
\hfill
\begin{subfigure}[t]{0.48\textwidth}
    \centering
    \includegraphics[width=\textwidth]{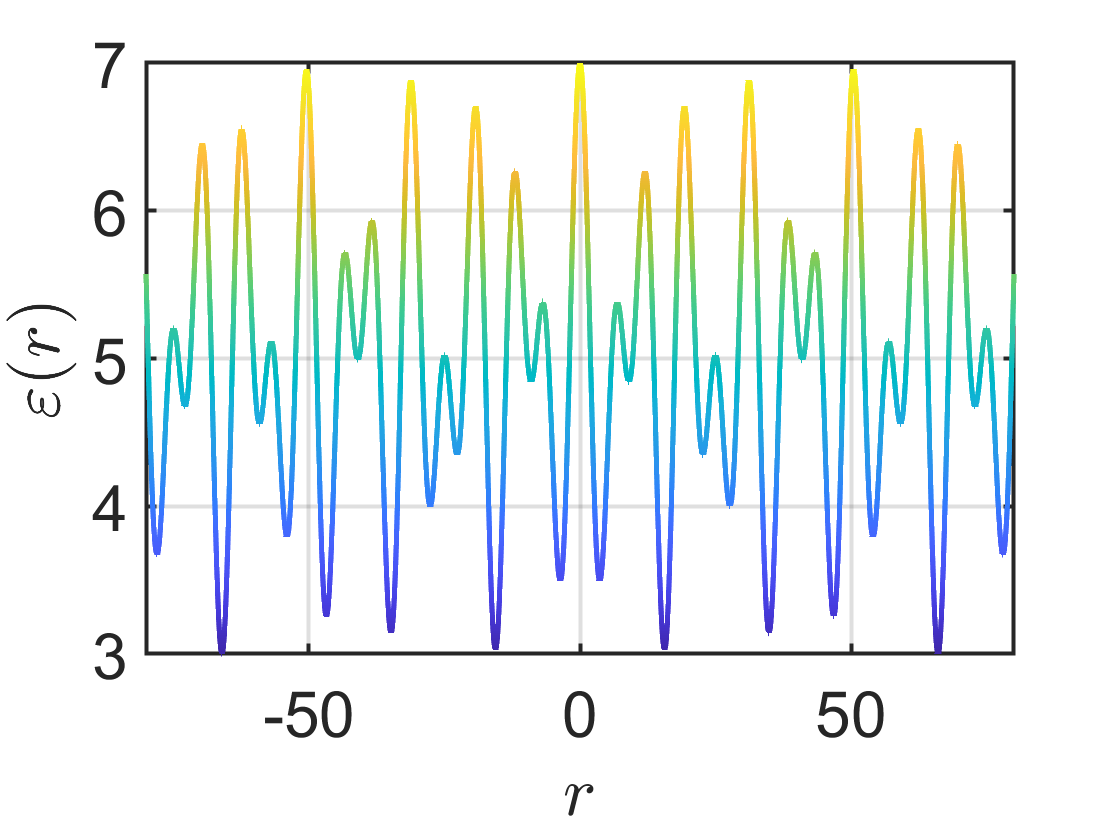}
    \caption{Projected 1D quasiperiodic coefficient.}
    \label{subfig:1D_quasi_coefficient}
\end{subfigure}
\caption{Generation of a 1D quasiperiodic coefficient from a 2D periodic function. (a) Periodic coefficient on $\mathcal T^2$ with the white line indicating the irrational projection direction. (b) Projected 1D quasiperiodic coefficient obtained by restriction the periodic coefficient along the irrational line. }
\label{fig:1Dto2D_coefficient_projection}
\end{figure}

For the 2D-to-4D case, the quasiperiodic coefficient is generated from the 4D periodic function
\begin{equation}
\mathcal{E}(\mathbf{x}) =
\cos(x_1)+\cos(x_2)+\sin(x_3)+\sin(x_4)+10,
\qquad
P=
\begin{bmatrix}
1 & \sqrt{2} & \sqrt{3} & \sqrt{5} \\
\sqrt{7} & 1 & e^{-1} & e
\end{bmatrix}^{\mathsf T}.
\label{def:4D_per_evp_func}
\end{equation}
After solving the corresponding embedded eigenvalue problems, the low-dimensional eigenfunctions are reconstructed according to \eqref{sol: spectral projection} and \eqref{sol: 4D projected function}.

\begin{figure}[htbp]
\centering
\begin{subfigure}[b]{1\textwidth}
    \centering
    \includegraphics[width=0.7\textwidth]{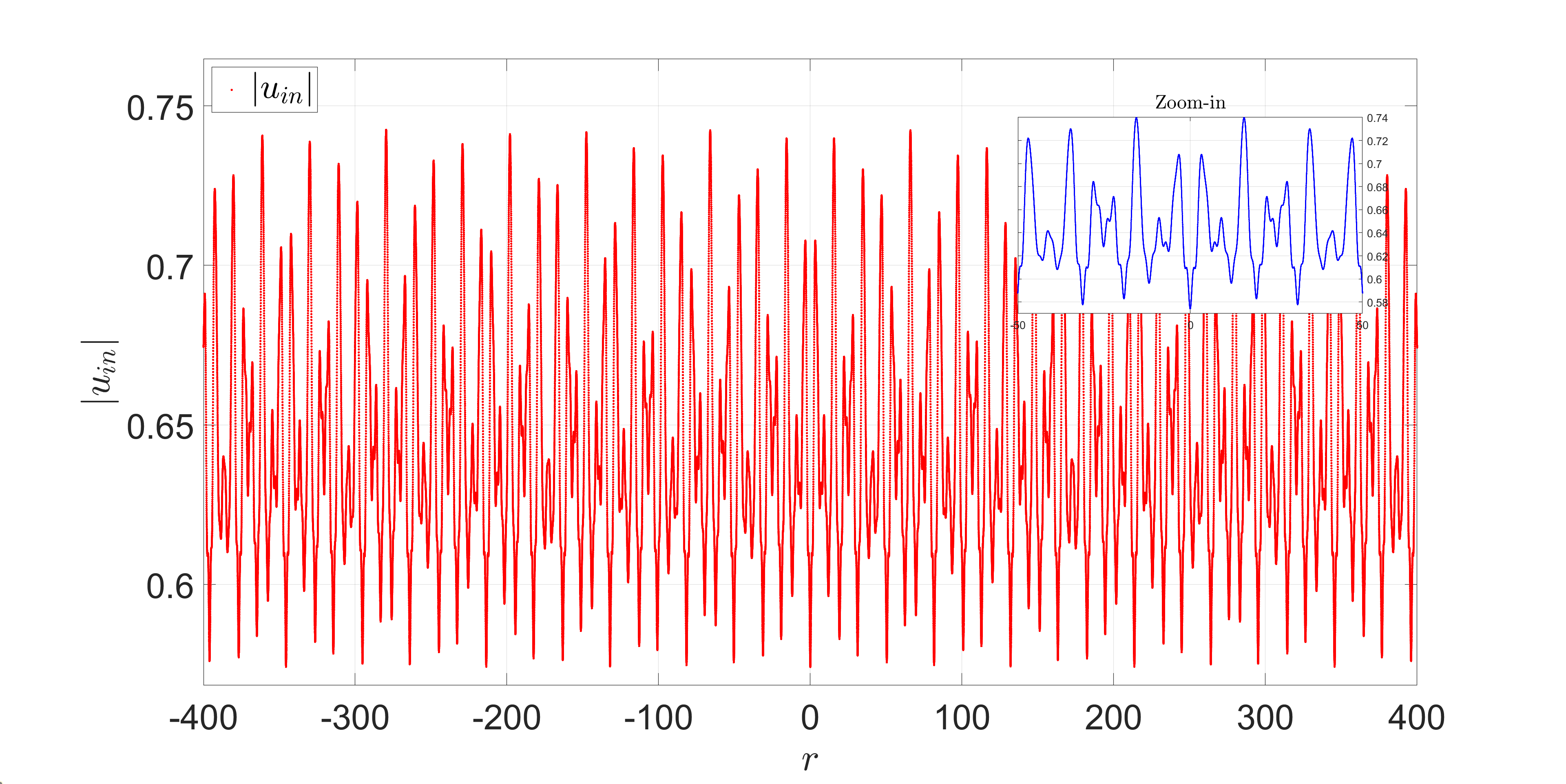}
    \caption{Reconstructed 1D projected eigenfunction $\widehat u(r)$ in the physical domain.}
\end{subfigure}
\hfill
\begin{subfigure}[t]{0.48\textwidth}
    \centering
    \includegraphics[width=\textwidth]{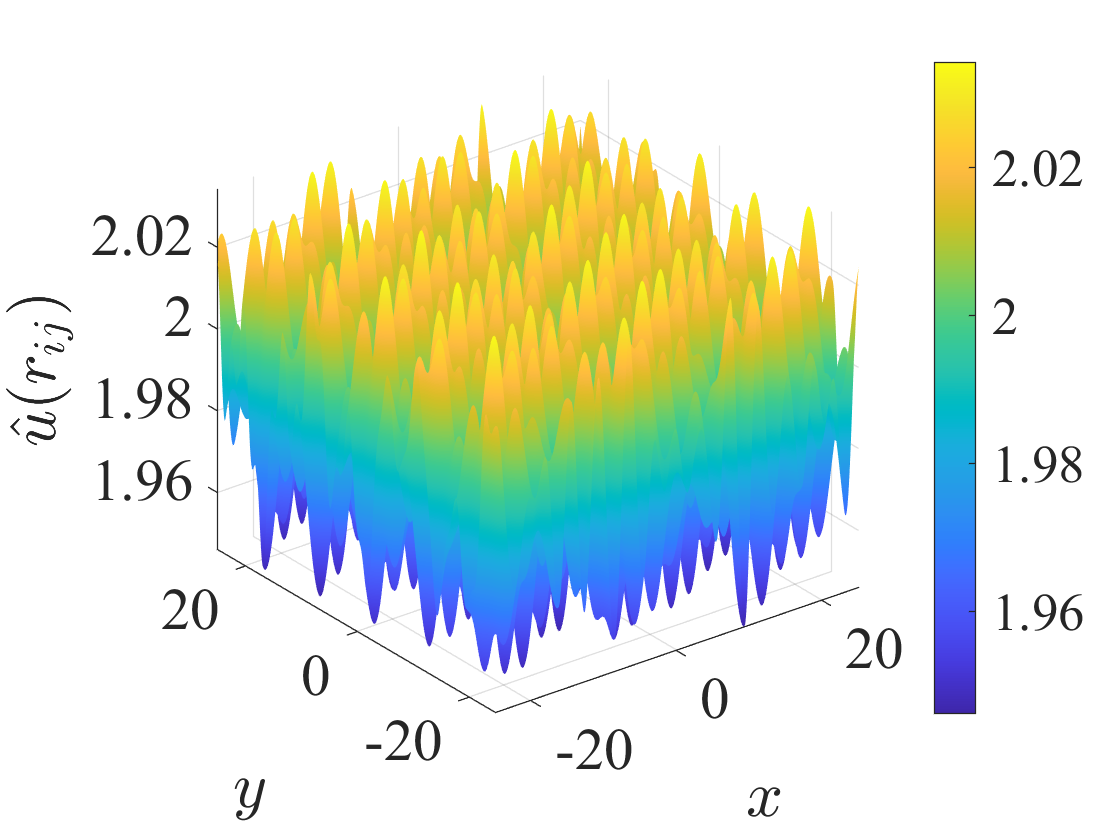}
    \caption{Surface of the reconstructed 2D projected eigenfunction $\widehat u(\mathbf r)$.}
\end{subfigure}
\begin{subfigure}[t]{0.48\textwidth}
    \centering
    \includegraphics[width=\textwidth]{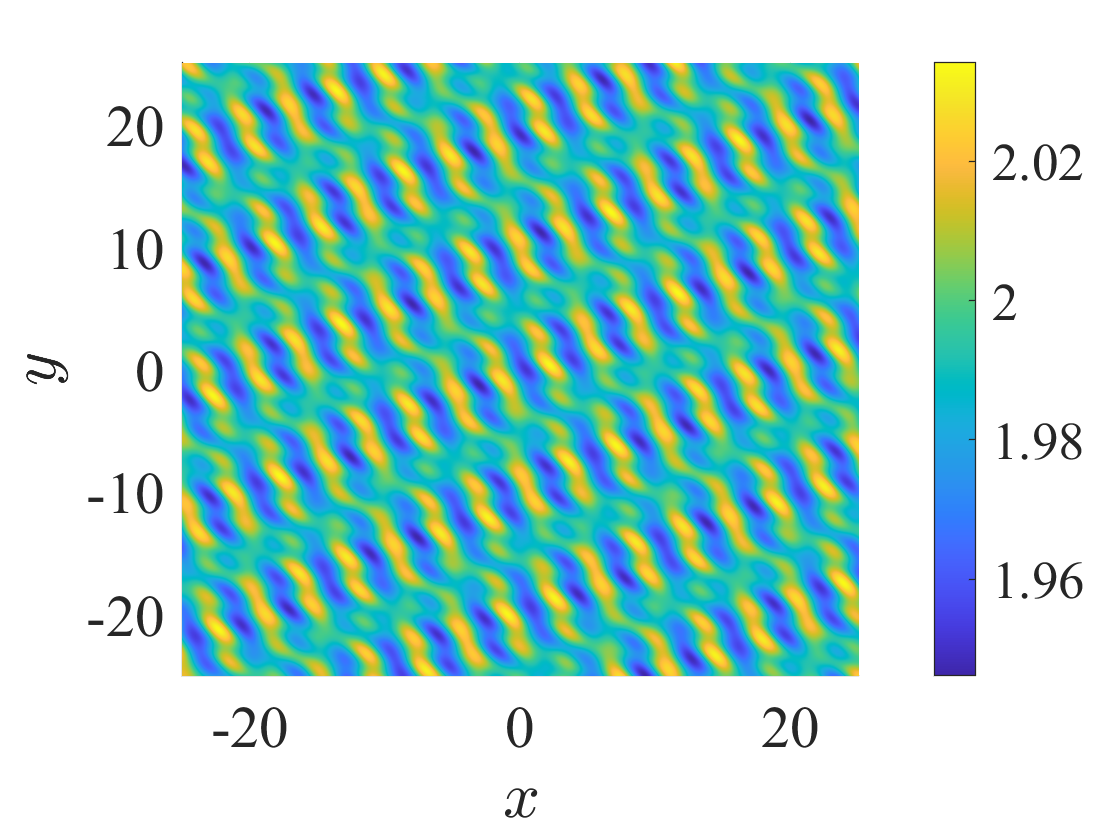}
    \caption{Planar view of the same reconstructed 2D projected eigenfunction.}
\end{subfigure}
\caption{Reconstructed projected eigenfunctions in the physical domain. (a) shows the 1D-to-2D case, where $\widehat{\lambda}(\mathbf{k})=0.492125$ with $\mathbf{k}=[0.35,0]^{\mathsf T}$. (b) and (c) show the same 2D-to-4D reconstructed eigenfunction in surface and planar views, corresponding to $\widehat{\lambda}(\mathbf{k})=0.025178$ with $\mathbf{k}=[0.5,0.2,0,0]^{\mathsf{T}}$, presenting its quasiperiodic spatial oscillations.}
\label{fig: projected field}
\end{figure}

Figure~\ref{fig: projected field} shows the reconstructed eigenfunctions in the physical domain. The oscillatory patterns demonstrate that periodic embedded eigenfunctions are transformed by the projection relation into quasiperiodic low-dimensional fields.

\subsubsection*{Example 2 (Validation of the pointwise RQ estimator)}

We next validate the pointwise RQ estimator developed in Section~\ref{section: Error analysis}. The main goal of this example is to examine whether the reconstructed low-dimensional function satisfies the original low-dimensional eigenvalue relation in an averaged local sense.

For each embedded eigenpair $(\lambda,U)$ computed from the projected periodic formulation, we reconstruct the corresponding low-dimensional function $\widehat{u}$. The pointwise RQs are then evaluated on a sampled physical grid using the cropped FD operators of the original low-dimensional problem. Their weighted expectation is compared with the projected eigenvalue $\lambda$.

In the 1D-to-2D case, $\widehat u(r)$ is reconstructed by \eqref{sol: spectral projection}. The cropped 1D FD operators are assembled according to \eqref{total: finite matrix derive}, the pointwise RQs are computed by \eqref{eq: pointwise ratio}, and the weights are given by \eqref{set: probability density}.

In the 2D-to-4D case, $\widehat u(\mathbf r)$ is reconstructed by \eqref{sol: 4D projected function}. The pointwise RQs are computed from the cropped 2D FD operators as in \eqref{def: statistic r_ij}, with the corresponding weights defined by \eqref{def: RQ of 2Dto4D}.

\begin{figure}[htbp]
\centering
\begin{subfigure}[b]{0.48\textwidth}
    \centering
    \includegraphics[width=\textwidth]{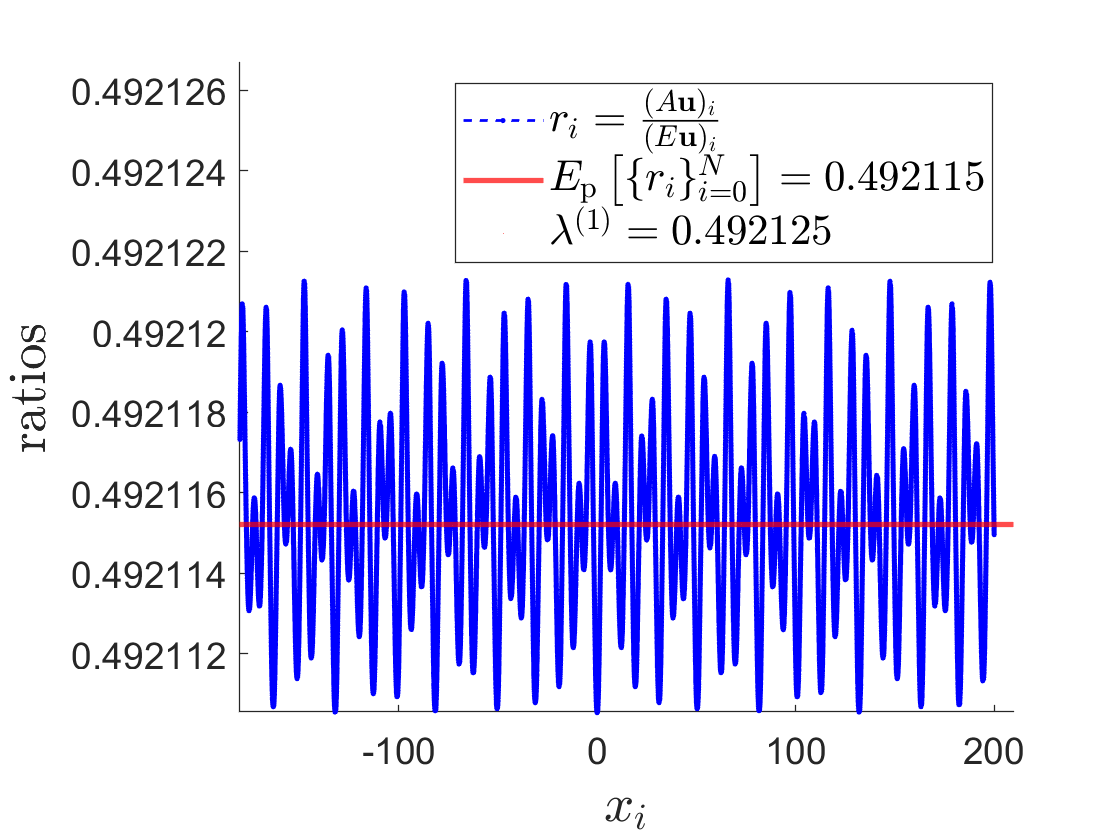}
    \caption{Pointwise RQs for the 1D-to-2D case.}
\end{subfigure}
\hfill
\begin{subfigure}[b]{0.48\textwidth}
    \centering
    \includegraphics[width=\textwidth]{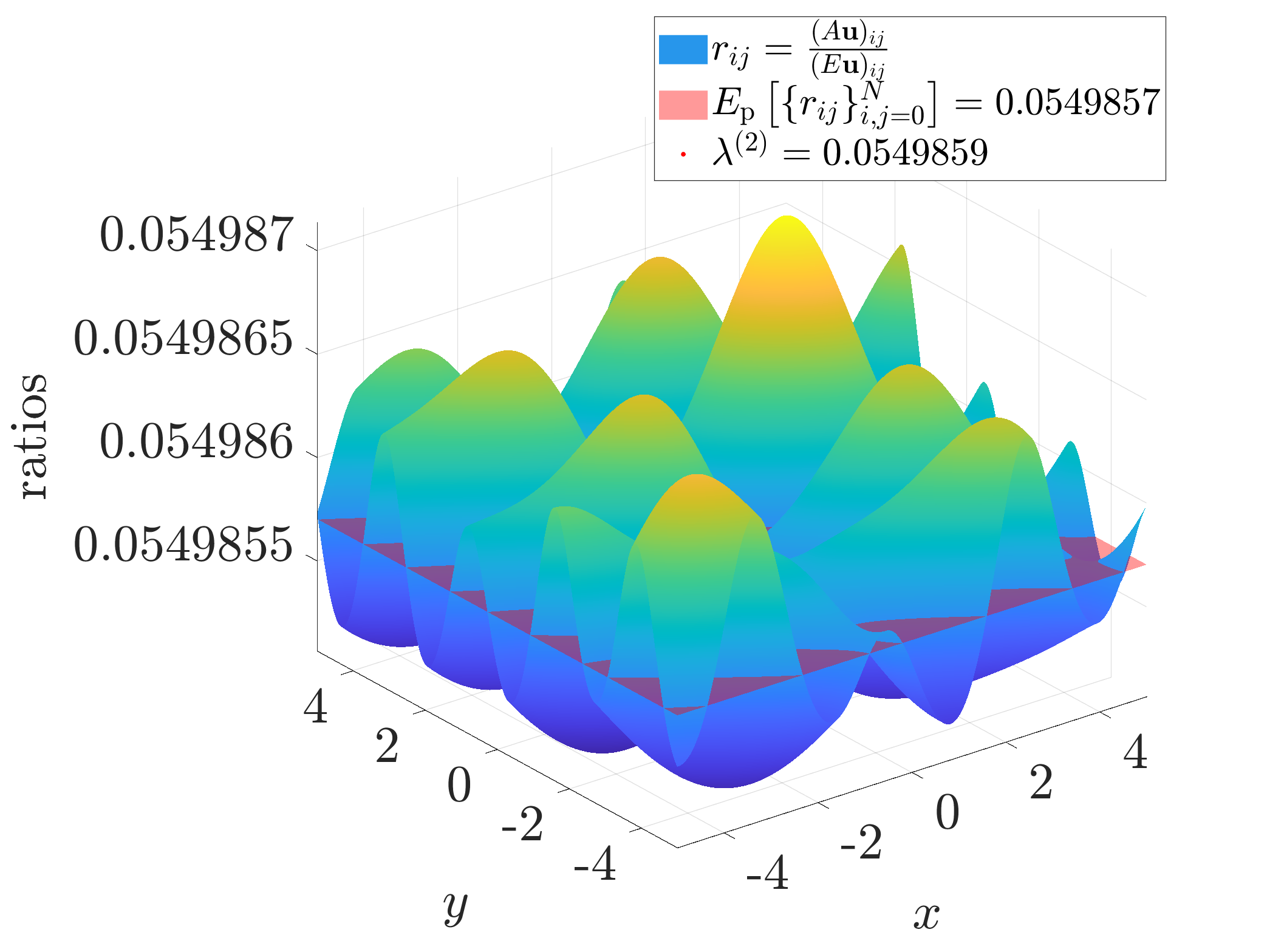}
    \caption{Pointwise RQs for the 2D-to-4D case.}
\end{subfigure}
\caption{Pointwise RQs for the projected eigenvalue problems. (a) Pointwise RQs and their weighted expectation for the 1D-to-2D case, evaluated on $4\times 10^4$ physical sampling points. (b)Pointwise RQs and their weighted expectation for the 2D-to-4D case, evaluated on $10^6$ physical sampling points.}
\label{fig:local_RQs}
\end{figure}

Figure~\ref{fig:local_RQs} shows the pointwise RQs for the two projected eigenvalue problems, with $\mathbf{k}=[0.35,0]^{\mathsf T}$ for the 1D-to-2D case and $\mathbf{k}=[0.5,0.2,0,0]^{\mathsf T}$ for the 2D-to-4D case. To ensure the stability of the weighted expectations, we use $4\times 10^4$ physical sampling points for the 1D-to-2D case and $10^6$ physical sampling points for the 2D-to-4D case. In both cases, the weighted expectation is close to the projected eigenvalue, indicating that the reconstructed eigenfunction satisfies the low-dimensional eigenvalue relation in an averaged local sense. The fluctuations in the 2D-to-4D case are more visible, which is expected since the sampled quasiperiodic structure is more complex.

To quantify the approximation accuracy, we record the errors
\begin{equation}
e_{\mathrm{RQ}}^{(1)}
=
\left|
\mathbb{E}_{\mathrm p}[r_i]-\lambda^{(1)}
\right|,
\qquad
e_{\mathrm{RQ}}^{(2)}
=
\left|
\mathbb{E}_{\mathrm p}[r_{ij}]-\lambda^{(2)}
\right|,
\label{eq:RQ_validation_errors}
\end{equation}
for the 1D-to-2D and 2D-to-4D cases, respectively. The corresponding numerical
values for different numbers of basis functions are reported in
Table~\ref{tab:RQ_validation}.

\begin{table}[htbp]
\centering
\renewcommand{\arraystretch}{1.5}
\caption{Errors $e_{\mathrm{RQ}}^{(1)}$ and $e_{\mathrm{RQ}}^{(2)}$ between the weighted expectation of pointwise RQs and the
projected eigenvalue for 1D-to-2D and 2D-to-4D cases with respect to the number of Fourier bases $N$, respectively.}
\begin{tabular}{|c|c|c|c|c|c|c|c|}
\hline
$N^{(1)}$ (1D-to-2D) & 8 & 12 & 16 & 20 & 24 & 28 & 32 \\ \hline
$e_{\mathrm{RQ}}^{(1)}$ & 1.4e-1 & 8.1e-3 & 3.0e-3 & 1.5e-3 & 7.1e-4 & 5.1e-4 & 4.5e-4 \\ \hline
$N^{(2)}$ (2D-to-4D) & 4 & 6 & 8 & 10 & 12 & 14 & 16 \\ \hline
$e_{\mathrm{RQ}}^{(2)}$ & 2.7e-2 & 2.7e-2 & 3.6e-3 & 2.4e-3 & 1.8e-3 & 9.9e-4 & 1.2e-3 \\ \hline
\end{tabular}
\label{tab:RQ_validation}
\end{table}

The results show that the weighted expectations $\mathbb{E}_{\mathrm p}[r_i]$ and $\mathbb{E}_{\mathrm p}[r_{ij}]$ agree well with the projected eigenvalue $\lambda^{(1)}$ and $\lambda^{(2)}$, respectively. The errors decrease as the number of basis functions increases, supporting the consistency of the pointwise RQ estimator. The slight saturation or nonmonotonicity at larger basis sizes is expected, since the validation also involves finite sampling and FD evaluation on the physical grid.

\subsection*{Example 3 (Smoothing and validation for discontinuous quasiperiodic media)}
\label{subsec:regularity_smoothing}

In this example, we examine the effect of smoothing on the projected spectral approximation for a discontinuous quasiperiodic coefficient. Starting from the original discontinuous coefficient, we introduce smoothed embedded models with different smoothing widths. For each smoothed model, the reconstructed low-dimensional function and its pointwise RQ estimator are then evaluated against the original discontinuous low-dimensional operator, in order to assess whether they still provide an effective approximate eigenpair of the original problem.

We consider the 2D-to-4D projected setting with $\mathbf r=(r_1,r_2)^\mathsf T\in\mathbb R^2$ and $\mathbf x=P\mathbf r\in\mathcal T^4$, where
\begin{equation}
P=
\begin{bmatrix}
1 & \sqrt3 & \sqrt5 & \sqrt{11}\\
\sqrt2 & 1 & \sqrt7 & \sqrt{13}
\end{bmatrix}^{\mathsf T}.
\label{eq:regularity_projection_matrix}
\end{equation}
We use the following embedded coefficient as a reference smooth model
\begin{equation}
\mathcal E_{\mathrm s}(\mathbf x)
=
\exp \bigl(
0.3\cos x_1 + 0.2\sin x_2
\bigr)
+ 0.2\cos x_3 - 0.3\sin x_4 + 4.
\label{eq:smooth_regularity_coefficient}
\end{equation}
The discontinuous coefficient is generated from a shell-energy-type periodic coefficient
\begin{equation}
\mathcal E_{\mathrm d}(\mathbf x) =
\begin{cases}
\varepsilon_B,
&
\cos x_1+\cos x_2+\cos x_3+\cos x_4 > 1, \\
\varepsilon_A,
&
\text{otherwise},
\end{cases} \qquad
\varepsilon_A=1, \quad \varepsilon_B=12.
\label{eq:discontinuous_regularity_coefficient}
\end{equation}

To visualize the discontinuous embedded coefficient, we present two complementary views in Figure~\ref{fig:4D_slices}. The left panel contains two representative 3D slices of the 4D periodic coefficient, obtained by fixing $x_4=0$ and $x_4=\pi/2$, respectively, with $(x_1,x_2,x_3)\in[-\pi,\pi]^3$. These two slices illustrate the two-phase interface structure of the embedded discontinuous coefficient from different sections. The right panel shows the projected coefficient on the 2D physical domain, where the irrational projection matrix induces a quasiperiodic interface pattern.

\begin{figure}[htbp]
\centering
\begin{subfigure}[t]{0.48\textwidth}
    \centering
    \includegraphics[width=\textwidth]{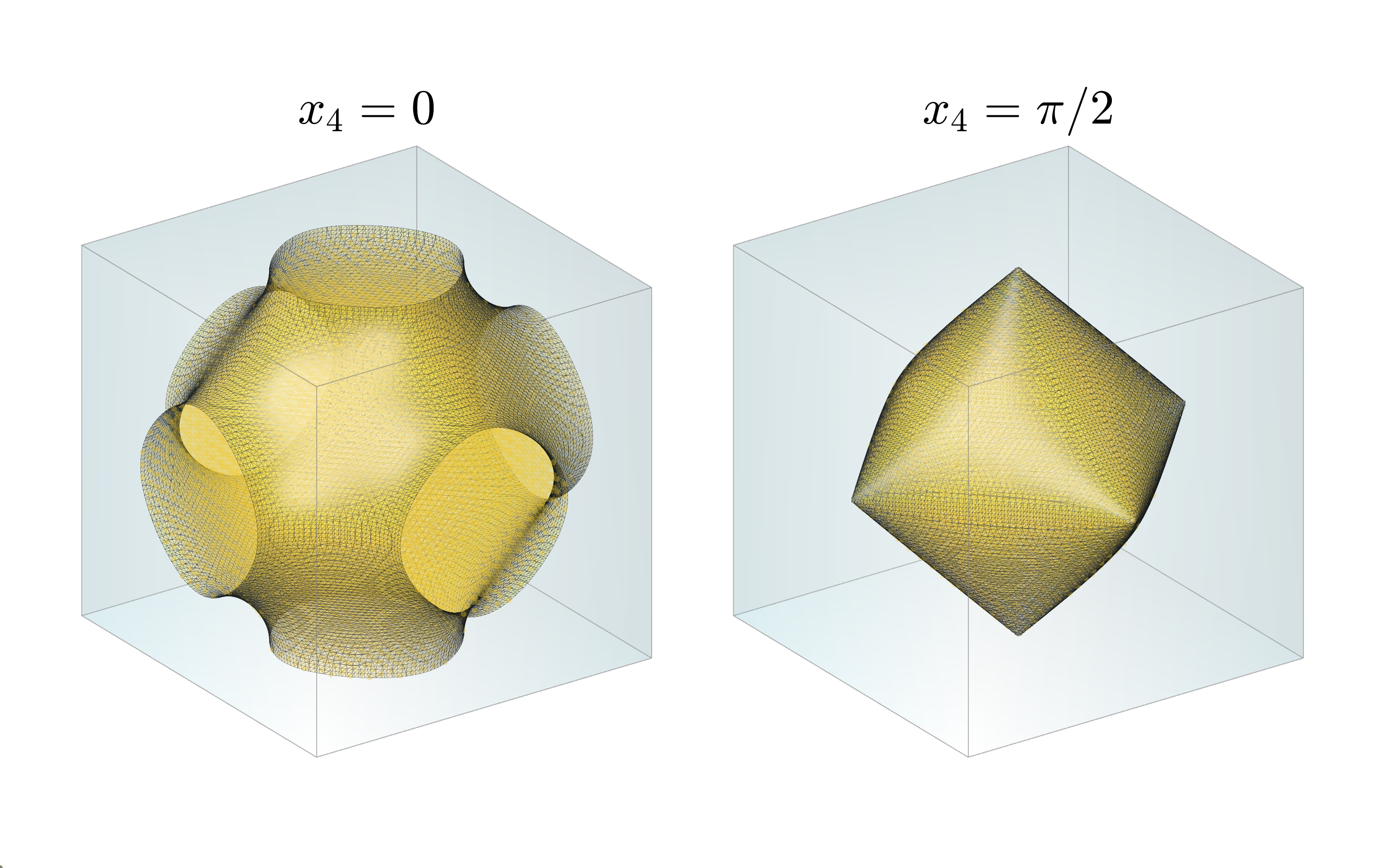}
    \caption{Representative 3D slices at $x_4=0$ and $x_4=\pi/2$, with $(x_1,x_2,x_3)\in[-\pi,\pi]^3$.}
\end{subfigure}
\hfill
\begin{subfigure}[t]{0.48\textwidth}
    \centering
    \includegraphics[width=\textwidth]{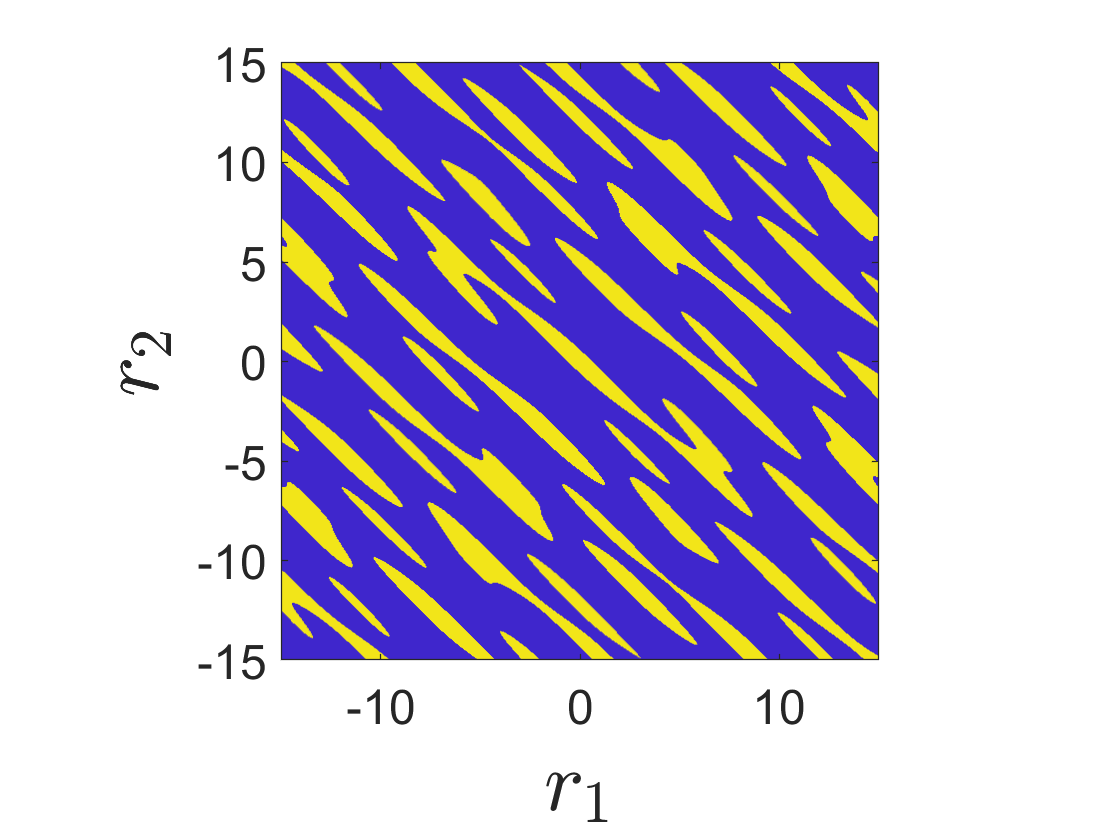}
    \caption{Projected 2D physical slice showing the quasiperiodic interface structure.}
\end{subfigure}
\caption{Visualization of the discontinuous embedded coefficient. (a) shows two representative 3D slices of the 4D periodic coefficient at $x_4=0$ and $x_4=\pi/2$, while (b) shows the projected 2D quasiperiodic coefficient in the physical domain.}
\label{fig:4D_slices}
\end{figure}

To reduce interface-induced Gibbs oscillations, we introduce a Tanh-type smoothing \cite{xiao2005simple} of $\mathcal E_{\mathrm d}$. Let $d(\mathbf x)$ denote the signed distance to the interface, with $d(\mathbf x)>0$ inside the phase $\varepsilon_B$ and $d(\mathbf x)<0$ in the background phase $\varepsilon_A$. The smoothed coefficient is defined by
\begin{equation}
\mathcal E_\tau(\mathbf x)
=
\frac{\varepsilon_A+\varepsilon_B}{2} +
\frac{\varepsilon_B-\varepsilon_A}{2}
\tanh \left( \frac{d(\mathbf x)}{\tau} \right),
\label{eq:smoothed_regularity_coefficient}
\end{equation}
where $\tau>0$ is the smoothing width.

To visualize how smoothing suppresses high-frequency Fourier components, we group the Fourier modes into shells
\begin{equation}
\mathcal J_q
=
\{
\boldsymbol\xi\in\mathcal J:
\|\boldsymbol\xi\|_\infty=q
\},
\qquad
q=0,1,\ldots,N,
\label{def: Fourier mode shells}
\end{equation}
and compute the shell-wise RMS Fourier amplitude \cite{mininni2009scale}
\begin{equation}
A_q^{\mathrm{rms}}(\mathcal E)
=
\left(
\frac{1}{\#\mathcal J_q}
\sum_{\boldsymbol\xi\in\mathcal J_q}
|\mathcal E_{\boldsymbol\xi}|^2
\right)^{1/2}.
\label{eq:shellwise_fourier_energy}
\end{equation}
This quantity measures the average Fourier strength on each frequency shell and is used to compare the decay behavior of $\mathcal E_{\mathrm s}$, $\mathcal E_{\mathrm d}$, and $\mathcal E_\tau$.

Figure~\ref{fig:coefficient_fourier_decay} shows that the smooth coefficient has the fastest Fourier decay, while the discontinuous coefficient retains a pronounced high-frequency tail. After smoothing, the high-frequency amplitudes are reduced, indicating that the smoothed coefficient is more favorable for Fourier spectral approximation.

We next examine whether the smoothed projected models preserve the spectral information of the original discontinuous quasiperiodic problem. For each smoothing width $\tau$, we solve the corresponding projected eigenvalue problem and reconstruct the low-dimensional eigenfunction $\widehat u_\tau$. The reconstructed function is then substituted into the original discontinuous low-dimensional operator to compute the pointwise RQs

\begin{equation}
r_s^{(\tau)} =
\frac{
(\widehat A\widehat{\mathbf u}_\tau)_s
}{
(\widehat B_{\mathrm d}\widehat{\mathbf u}_\tau)_s
},
\label{eq:smoothed_validation_rq}
\end{equation}
where $\widehat B_{\mathrm d}$ is assembled from the original discontinuous
coefficient.

The corresponding weighted expectation
$ \widehat{\lambda}_{\tau} = \mathbb E_{\mathrm p} \left[ \{r_s^{(\tau)}\} \right]$
is used as the low-dimensional spectral estimator evaluated under the original discontinuous model. To measure the local spread of the pointwise RQs around this estimator, we also compute the weighted standard deviation
\begin{equation}
\sigma_{\mathrm{RQ}}(\tau)
= \left(
\mathbb E_{\mathrm p}
\left[
\left(
r_s^{(\tau)} - \widehat{\lambda}_{\tau}
\right)^2
\right]
\right)^{1/2}.
\label{eq:smoothed_rq_std}
\end{equation}

In this way, the smoothed eigenfunction is evaluated against the original discontinuous model rather than only against the smoothed coefficient.
\begin{figure}[htbp]
\centering
\begin{subfigure}[t]{0.48\textwidth}
    \centering
    \includegraphics[width=\textwidth]
    {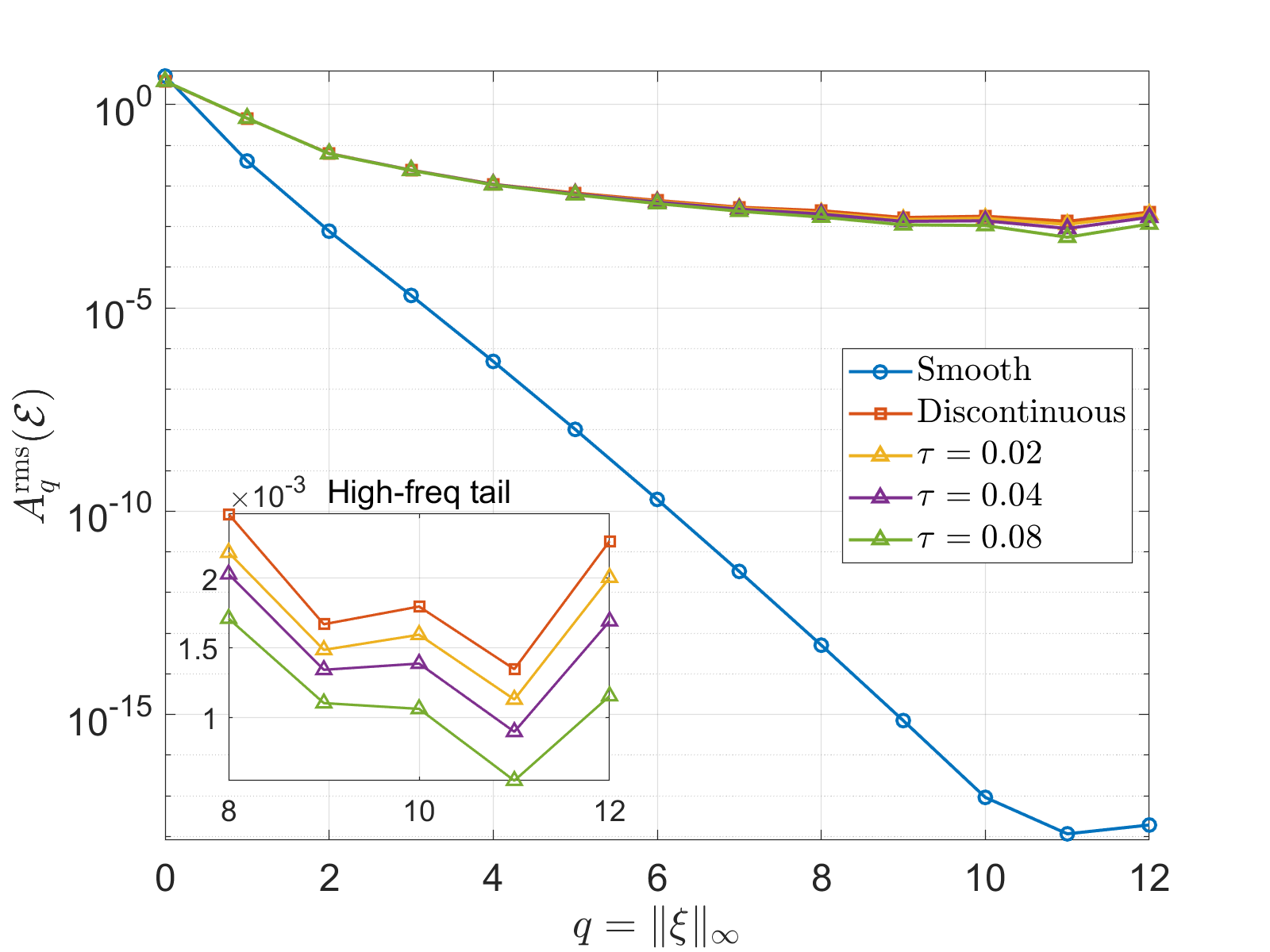}
    \caption{x-axis denotes the Fourier mode shell in \eqref{def: Fourier mode shells} and y-axis denotes the shell-wise RMS Fourier amplitude in \eqref{eq:shellwise_fourier_energy}.
    }
    \label{fig:coefficient_fourier_decay}
\end{subfigure}
\hfill
\begin{subfigure}[t]{0.48\textwidth}
    \centering
    \includegraphics[width=\textwidth]
    {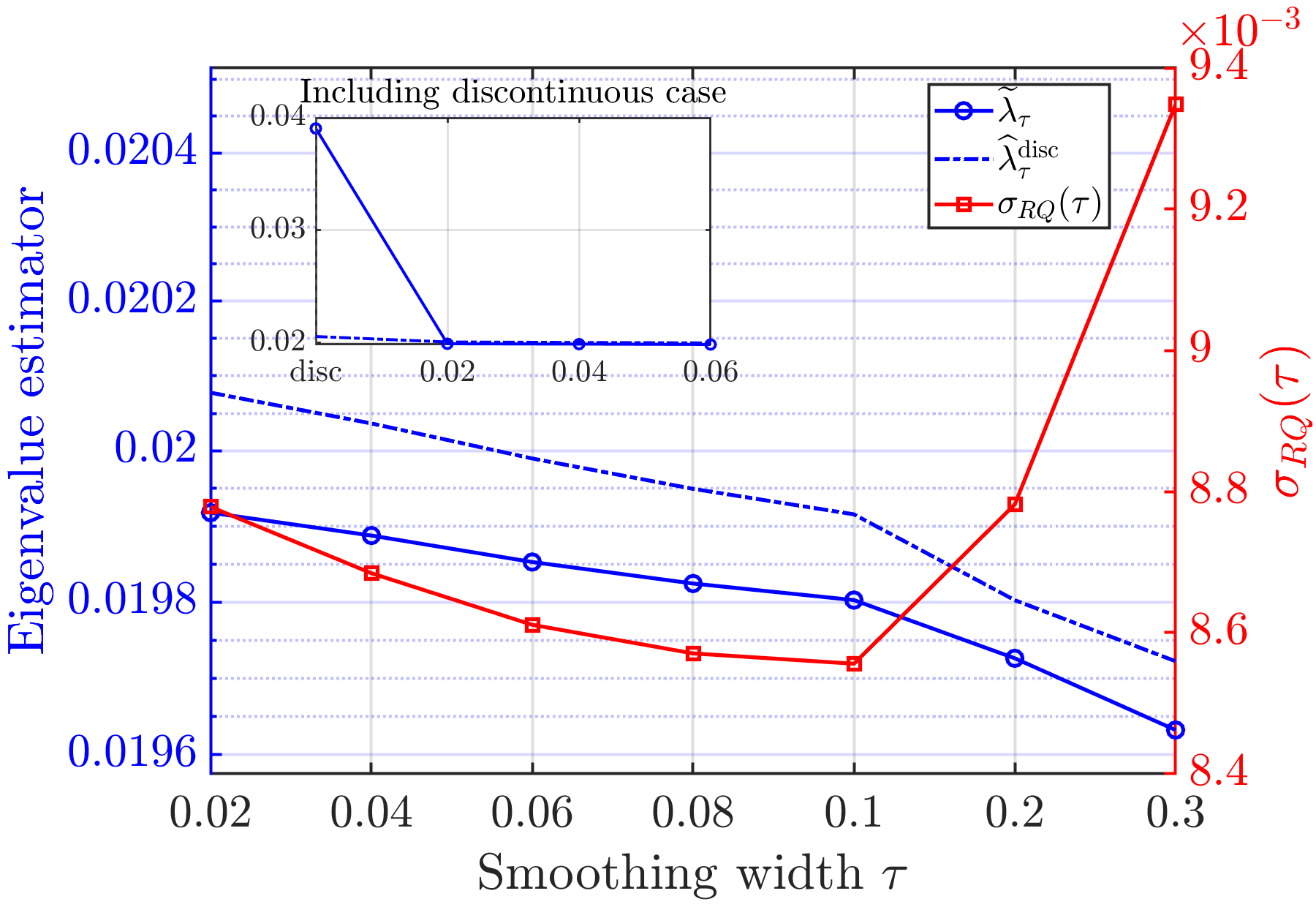}
    \caption{Left axis: eigenvalue estimator $\widehat{\lambda}_\tau$ (blue dashed line) vs. projected eigenvalue $\widetilde{\lambda}_\tau$ (blue line); right axis: weighted standard diviations (red line) of the pointwise RQs, for different smoothed width $\tau$, respectively.}
    \label{fig:smoothing_validation}
\end{subfigure}
\caption{Effect of smoothing on the spectral approximation of discontinuous quasiperiodic media. (a) Shell-wise RMS Fourier amplitudes of the smooth, discontinuous, and smoothed embedded coefficients. The inset highlights the high-frequency Fourier amplitudes of the smoothed coefficients. (b) Projected eigenvalues, low-dimensional RQ estimators, and weighted standard deviations of the pointwise RQs for different smoothing widths.}
\label{fig:smoothing_combined}
\end{figure}

Figure~\ref{fig:smoothing_validation} compares the projected eigenvalue $\widetilde{\lambda}_{\tau}$ of the 30th eigenmode at $\mathbf{k}=[0.35,0.2,0.1,0.5]^{\mathsf T}$ with the corresponding low-dimensional RQ estimator
$\widehat{\lambda}_{\tau} = \mathbb E_{\mathrm p_\tau}\left[\{r_s^{(\tau)}\}\right]$.
Here $r_s^{(\tau)}$ is computed by substituting the reconstructed function $\widehat u_\tau$ into the original discontinuous low-dimensional operator. The weighted standard deviation $\sigma_{\mathrm{RQ}}(\tau)$ is also plotted to measure the local spread of the pointwise RQs around $\widehat{\lambda}_{\tau}$.

The inset shows that, for the original discontinuous model, the weighted expectation of the local RQs differs noticeably from the projected eigenvalue. After smoothing is introduced, this discrepancy is rapidly reduced, indicating that smoothing improves the compatibility between the projected eigenfunction and the low-dimensional RQ estimator. However, the weighted standard deviation does not decrease monotonically with the smoothing width. This suggests that a larger smoothing width is not necessarily preferable: the smoothing should be chosen to suppress interface-induced high-frequency oscillations while still preserving the spectral information of the original discontinuous model.

\subsubsection*{Example 4 (Band folding and projected spectral branches)}
\label{subsec:band_unfolding_fractal}

In this example, we compare the band structures obtained from rational approximant supercell methods and from the proposed projected spectral method. The purpose is to illustrate how the projected formulation reconstructs the intrinsic quasiperiodic spectral branches without introducing the repeated band folding caused by artificial periodic supercells.

Since a quasiperiodic structure has no finite primitive cell, supercell methods replace the irrational projection by rational periodic approximants. The large supercell leads to a small artificial Brillouin zone, so the spectral branches are repeatedly folded and become densely overlapped. Conventional unfolding methods recover effective bands from such folded supercell modes using spectral weights~\cite{zhang2022unfolded}. In contrast, the present projected spectral method reconstructs the quasiperiodic spectral branches directly from the higher-dimensional embedded eigenpairs through the projection relation and pointwise RQ estimators.

We first consider a 1D quasiperiodic medium generated from a two-dimensional
periodic embedding. Let
$\gamma=\frac{\sqrt{5}-1}{2}$ be the inverse golden ratio and
$\theta=\arctan\gamma$. The projection relation is given by
\begin{equation}
    \mathbf{x}=Pr\in\mathcal{T}^2,
    \qquad
    P=(\cos\theta,\sin\theta)^{\mathsf T},
    \qquad
    \varepsilon(r)=\mathcal{E}(Pr).
\end{equation}
We take the smooth periodic embedded coefficient
\begin{equation}
\mathcal{E}(x_1,x_2)
=
\varepsilon_0
+
\exp\!\left(
\alpha\cos x_1+\beta\sin x_2
\right),
\qquad
\alpha=0.5,\quad
\beta=1.3,\quad
\varepsilon_0=2.
\label{eq:smooth_2D_embedding_coefficient}
\end{equation}

The corresponding 1D coefficient
$\varepsilon(r)=\mathcal{E}(Pr)$ is quasiperiodic because the projection
direction has irrational slope determined by $\gamma$.

For the supercell computation, the irrational projection direction is replaced
by Fibonacci rational approximants, which generate periodic approximants of the
quasiperiodic medium on finite supercells. The resulting supercell bands are
then compared with the spectral branches reconstructed by the 1D-to-2D
projected spectral method.

For the supercell computation, the irrational entries in $P$ are replaced by
rational approximants, which generate a periodic approximant of the
quasiperiodic function on a finite 1D supercell. The folded supercell bands are
computed along the path
$$
\Gamma\rightarrow X,
\qquad
\mathbf{k}_{1D}=k_x,
\qquad
k_x\in\left[0,\frac{\pi}{T}\right],
$$
where $T$ is the supercell period.

For the projected spectral method, we use the lifted Bloch path $\Gamma^{\prime}\rightarrow X^{\prime}$,  where the lifted wave vector satisfies
$$
P^{\mathsf T}\mathbf{k}_{2D} = \mathbf{k}_{1D}.
$$
The projected spectral branches are reconstructed from the weighted
expectations of the pointwise RQs associated with the projected eigenfunctions.

We next consider the 2D-to-4D projected problem. The projection relation is
$$
\mathbf{x}=P\mathbf{r}\in\mathcal{T}^{4},
\qquad
\varepsilon(\mathbf{r})=\mathcal{E}(P\mathbf{r}),
$$
where $P$ has the same definition as that in \eqref{eq:regularity_projection_matrix}. The embedded periodic coefficient is chosen as
\begin{equation}
\mathcal{E}(\mathbf{x}) =
\varepsilon_0 + \alpha\bigl(\cos x_1+\cos x_2\bigr) + \beta\bigl(\cos x_3+\cos x_4\bigr) + \gamma\cos(x_1+x_3),
\label{eq:band_unfolding_coefficient}
\end{equation}
with $\varepsilon_0=8, \, \alpha=1.2, \, \beta=1.0, \, \gamma=0.8.$

For the supercell computation, the irrational entries in $P$ are replaced by rational approximants, which generate a periodic approximant on a finite 2D supercell. The folded supercell bands are computed along the path
$$
\Gamma\rightarrow M,
\qquad
\mathbf{k}_{2D}=(k_x,0),
\qquad
k_x\in\left[0,\frac{\pi}{T}\right],
$$
where $T$ is the supercell period in the $x$-direction.

The resulting supercell spectrum again exhibits dense folded branches due to the repeated projection of higher-dimensional spectral information into the reduced Brillouin zone of the periodic approximant.

For the projected spectral method, we use the corresponding lifted 4D Bloch path $\Gamma^{\prime}\rightarrow M^{\prime},$ where each lifted wave vector $\mathbf{k}_{4D}$ satisfies
$$
P^{\mathsf T}\mathbf{k}_{4D} = \mathbf{k}_{2D}.
$$
We then compute the first $N_{\mathrm{eig}}$ eigenvalues of the embedded 4D
eigenvalue problem for each sampled wave vector on
$\Gamma^{\prime}\rightarrow M^{\prime}$.

The projected spectral branches are reconstructed from the corresponding higher-dimensional eigenpairs through the projection relation and the pointwise RQ estimators. Therefore, the resulting spectral curves preserve the intrinsic quasiperiodic wave-vector information carried by the lifted Bloch path, instead of repeatedly folding the spectrum into an artificial low-dimensional Brillouin zone.

\begin{figure}[htbp]
\centering
\begin{subfigure}[t]{0.48\textwidth}
    \centering
    \includegraphics[width=\textwidth]{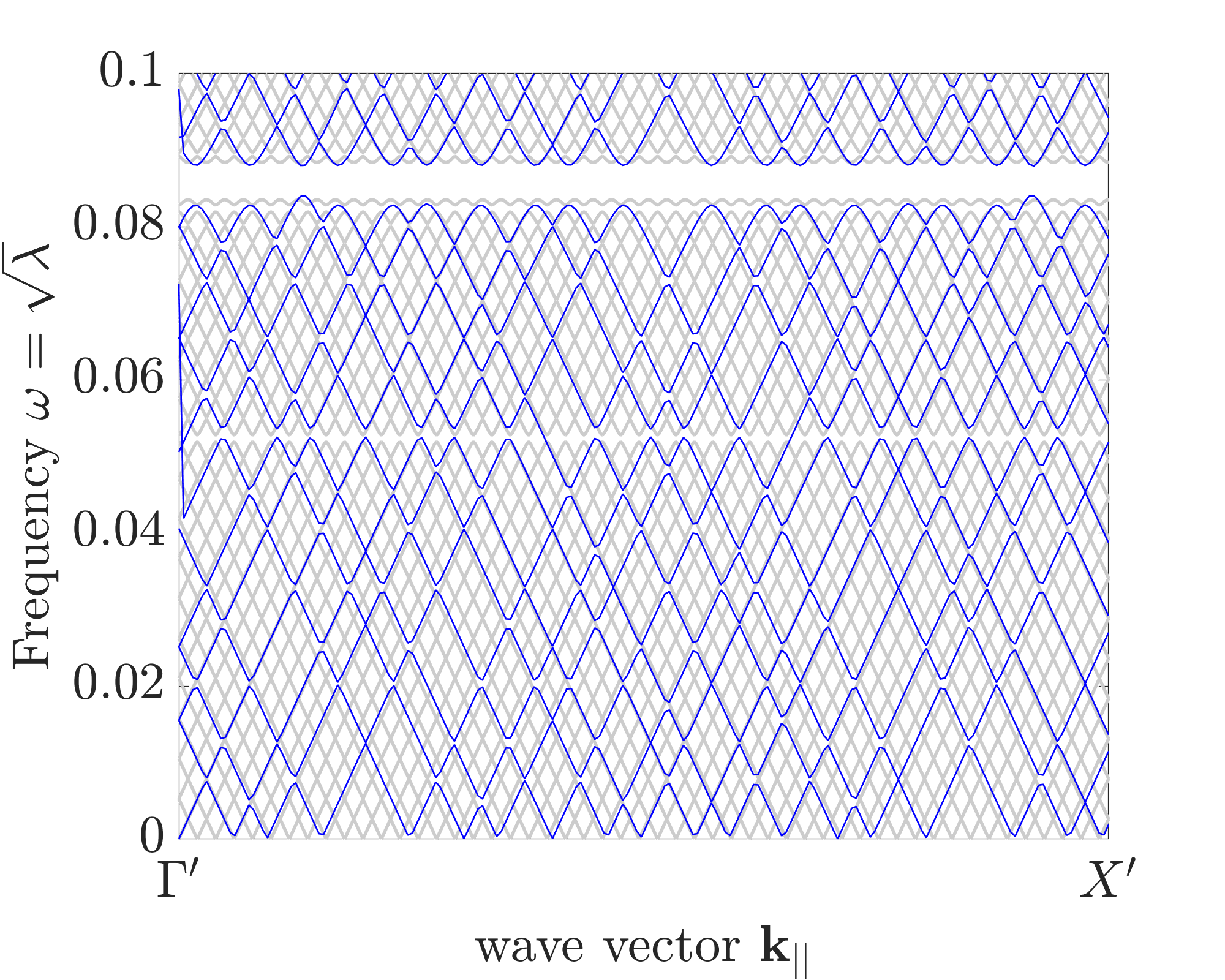}
    \caption{1D quasiperiodic function.}
    \label{fig:band_unfolding_1Dto2D_fibonacci}
\end{subfigure}
\hfill
\begin{subfigure}[t]{0.48\textwidth}
    \centering
    \includegraphics[width=\textwidth]{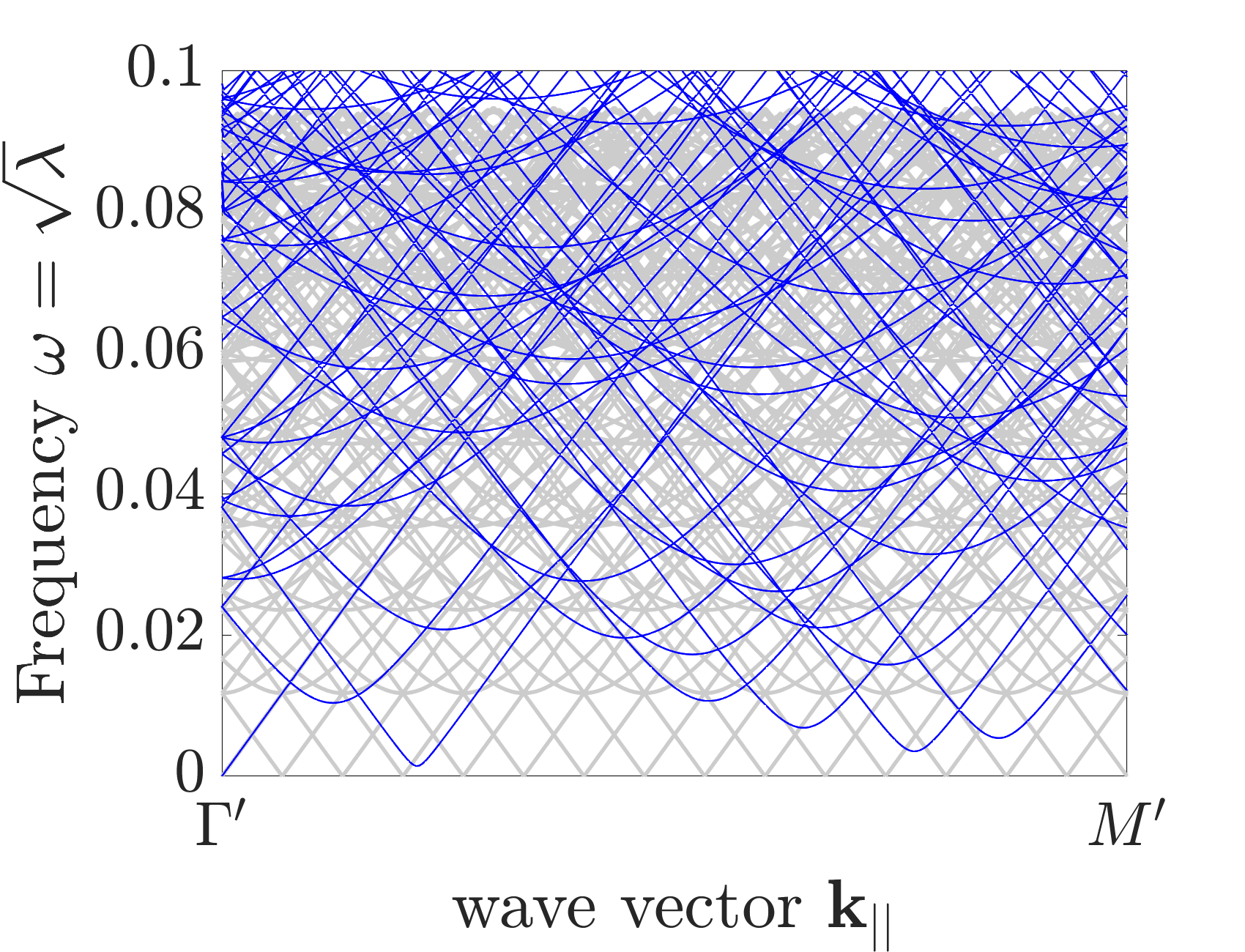}
    \caption{2D quasiperiodic function.}
    \label{fig:band_unfolding_2Dto4D_projection}
\end{subfigure}
\caption{Comparison of supercell bands and projected spectral branches. The gray curves show the folded bands produced by finite periodic supercell approximants, while the blue curves are obtained from the weighted expectations of the local RQs computed from the projected spectral method, representing the corresponding quasiperiodic spectral branches.}
\label{fig:band_unfolding_comparison}
\end{figure}

Figures~\ref{fig:band_unfolding_1Dto2D_fibonacci} and \ref{fig:band_unfolding_2Dto4D_projection} show that the projected spectral branches agree well with the dominant structures contained in the folded supercell spectra. Meanwhile, the projected formulation provides a much cleaner representation of the quasiperiodic spectral branches, since it avoids the dense repeated folding produced by finite supercell approximations. Some additional weak branches may also appear in the projected spectra. These branches can be interpreted as spurious or weakly physical bands caused by the finite Fourier truncation and the projection of high-dimensional modes whose low-dimensional reconstructed fields have small or unstable physical contributions.

Therefore, the proposed projected spectral method should not be interpreted as a band-unfolding procedure in the conventional supercell sense. Instead, it provides a direct reconstruction framework for quasiperiodic spectral branches based on higher-dimensional periodic embeddings and projected pointwise RQ estimators.

\section{Conclusions}
In this paper, we have presented a Fourier-pseudospectral framework in \cref{alg: projected spectral method} that circumvents the Diophantine approximation bottleneck by embedding 1D and 2D quasiperiodic Helmholtz eigenvalue problems into higher-dimensional periodic domains. Beyond solving the embedded periodic eigenvalue problems, we further reconstruct the computed eigenfunctions in the original low-dimensional physical space and represent the weighted expectation of pointwise RQs as the eigenvalue associated with the low-dimensional quasiperiodic system. From the modeling perspective, the resulting eigenpair $(\widehat{\lambda},\widehat u)$ provides a more appropriate approximation to the original low-dimensional quasiperiodic Helmholtz eigenvalue problem than the embedded eigenpair alone. The analysis of pointwise RQs further justifies this interpretation by showing that the weighted RQ expectation is a stable low-dimensional spectral estimator, and that the original quasiperiodic spectrum is contained in the closure of the spectra generated by the higher-dimensional embedded Bloch formulations. Numerical experiments on both continuous and discontinuous media confirm the consistency, stability, and efficiency of the proposed framework.

\bibliographystyle{plain}
\bibliography{research_paper}

\end{document}